%% file: main.tex
\documentclass{article}
\pdfoutput=1

\usepackage[left=3cm,right=3cm,top=3cm,bottom=3cm]{geometry}
\usepackage{dcolumn}
\usepackage{amsmath,amssymb,amsthm,mathtools}
\usepackage{tikz,tikz-cd}
\usepackage{bm}
\usepackage{framed} 

\newcolumntype{d}[1]{D{.}{.}{#1}}
\DeclareMathOperator{\low}{low}
\DeclareMathOperator{\USE}{use}
\DeclareMathOperator{\otp}{otp}
\DeclareMathOperator{\nf}{nf}

\newtheorem{theorem}{Theorem}[section]
\newtheorem{definition}[theorem]{Definition}
\newtheorem*{definition*}{Definition}
\newtheorem{lemma}[theorem]{Lemma}
\newtheorem*{lemma*}{Lemma}
\newtheorem{claim}[theorem]{Claim}
\newtheorem*{claim*}{Claim}
\newtheorem{corollary}[theorem]{Corollary}
\newtheorem*{corollary*}{Corollary}
\newtheorem{fact}[theorem]{Fact}
\newtheorem*{fact*}{Fact}
\newtheorem*{mainlemma1*}{Main Lemma 1}
\newtheorem*{mainlemma2*}{Main Lemma 2}

\begin{document}
\title{Nonlowness independent from frequent mind changes}
\author{Li Ling Ko\footnote{I would like to thank my advisor Prof. Peter
Cholak for his patience and guidance, without which this piece of work will
not be possible.} \footnote{Department of Mathematics, University of Notre
Dame, Hayes-Healy, Notre Dame, Indiana 46556, USA} \\ lko@nd.edu}
\date{\today}
\maketitle

\input{0-abstract.tex}

\vspace{1em}

\input{1-intro.tex}
\input{2-nonlow-ac.tex}

\input{3-low-alpha.tex}

\input{4-nonlow-alpha.tex}

\bibliographystyle{plain}           
\bibliography{bibliography}        

\end{document}

%% file: 0-abstract.tex
\begin{abstract}
It was recently shown that the computably enumerable (c.e.) degrees that embed the critical triple \cite{downey2007totally} and the $M_3$ lattice structure \cite{downey2015transfinite} are exactly those that change their minds sufficiently often. Therefore the embeddability strength of a c.e.\ degree has much to do with the degree's mind change frequency. Nonlowness is another common measure of degree strength, with nonlow degrees expected to compute more degrees than low ones. We ask if nonlowness and frequent mind changes are independent measures of strength. Downey and Greenberg (2015) claimed this to be true without proof, so we present one here. We prove the claim by building low and nonlow c.e.\ sets with an arbitrary number of mind changes. We base our proof on our direct construction of a nonlow $\mathrm{low}_2$ array computable set. Such sets were always known to exist, but also never constructed directly in any publication.
\end{abstract}

%% file: 1-intro.tex
\section{Introduction} \label{sec:introduction}

For many years one of the major themes in the study of the c.e.\ degrees is identifying the lattices that can be embedded below $0'$, and characterizing the degrees below which one can embed a given lattice. It has been known for some time that all distributive lattices can be embedded below any c.e.\ degree \cite{lachlan1972embedding}. As for non-distributive lattices, which are the lattices that contain either the $N_5$ or $M_3$ lattice (Figure~\ref{fig:lattices}) as sub-lattices, it is also known for some time that being able to embed $N_5$ is exactly the same as being non-contiguous \cite{downey1997contiguity}. However little was known about $M_3$ until 2015, when it was finally shown that a c.e.\ degree embeds $M_3$ if and only if the degree ``changes its mind frequently'' \cite{downey2015transfinite}. More precisely, a c.e.\ set $A$ needs to have a computable-approximation that ``changes its mind more than $\omega^\omega$-times'' before $A$ can embed $M_3$. We formally define what it means for a degree to ``change its mind $\alpha$-times'' in Section~\ref{sec:alpha}.\\

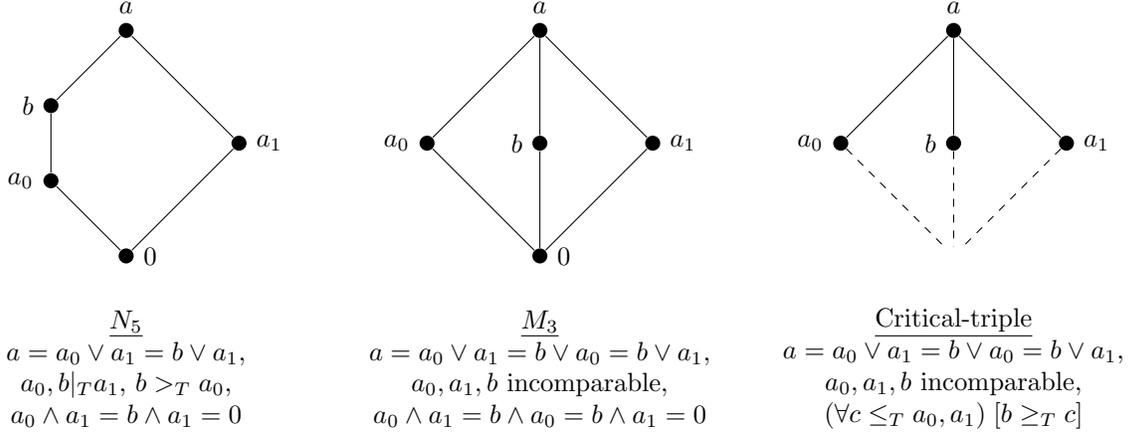
\begin{figure}[h!]
    \begin{tikzpicture}
        \tikzset{M3/.pic ={
            \node[circle, fill=black, inner sep=0pt, minimum size=0.2cm, label=$a$] (a) at (0, 0) {};
            \node[circle, fill=black, inner sep=0pt, minimum size=0.2cm, label=left:$a_0$] (a0) at (-1.5,-1.5) {};
            \node[circle, fill=black, inner sep=0pt, minimum size=0.2cm, label=left:$b$] (b) at (0,-1.5) {};
            \node[circle, fill=black, inner sep=0pt, minimum size=0.2cm, label=right:$0$] (0) at (0,-3) {};
            \node[circle, fill=black, inner sep=0pt, minimum size=0.2cm, label=right:$a_1$] (a1) at (1.5,-1.5) {};
            \draw [-] (a) -- (a0) -- (0) -- (a1) -- (a);
            \draw [-] (a) -- (b) -- (0);
            \node[align=center] at (0,-4.5) {\underline{$M_3$}\\ $a=a_0\vee a_1 =b\vee a_0 =b\vee a_1$,\\ $a_0,a_1,b$ incomparable,\\ $a_0\wedge a_1=b\wedge a_0 =b\wedge a_1=0$};
        }}
        \tikzset{N5/.pic ={
            \node[circle, fill=black, inner sep=0pt, minimum size=0.2cm, label=$a$] (a) at (0, 0) {};
            \node[circle, fill=black, inner sep=0pt, minimum size=0.2cm, label=left:$b$] (b) at (-1,-1) {};
            \node[circle, fill=black, inner sep=0pt, minimum size=0.2cm, label=left:$a_0$] (a0) at (-1,-2) {};
            \node[circle, fill=black, inner sep=0pt, minimum size=0.2cm, label=right:$0$] (0) at (0,-3) {};
            \node[circle, fill=black, inner sep=0pt, minimum size=0.2cm, label=right:$a_1$] (a1) at (1.5,-1.5) {};
            \draw [-] (a) -- (b) -- (a0) -- (0) -- (a1) -- (a);
            \node[align=center] at (0,-4.5) {\underline{$N_5$}\\ $a=a_0\vee a_1 =b\vee a_1$,\\ $a_0,b|_T a_1$, $b>_Ta_0$,\\ $a_0\wedge a_1=b\wedge a_1 =0$};
        }}
        \tikzset{triple/.pic ={
            \node[circle, fill=black, inner sep=0pt, minimum size=0.2cm, label=$a$] (a) at (0, 0) {};
            \node[circle, fill=black, inner sep=0pt, minimum size=0.2cm, label=left:$a_0$] (a0) at (-1.5,-1.5) {};
            \node[circle, fill=black, inner sep=0pt, minimum size=0.2cm, label=left:$b$] (b) at (0,-1.5) {};
            \node (0) at (0,-3) {};
            \node[circle, fill=black, inner sep=0pt, minimum size=0.2cm, label=right:$a_1$] (a1) at (1.5,-1.5) {};
            \draw [-] (a) -- (a0);
            \draw [-] (a) -- (a1);
            \draw [-] (a) -- (b);
            \draw [dashed] (a0) -- (0);
            \draw [dashed] (a1) -- (0);
            \draw [dashed] (b) -- (0);
            \node[align=left] (A) at (0,1) {#1};
            \node[align=center] at (0,-4.5) {\underline{Critical-triple} \\$a=a_0\vee a_1 =b\vee a_0 =b\vee a_1$,\\ $a_0,a_1,b$ incomparable,\\ $(\forall c\leq_T a_0,a_1)\; [b \geq_T c]$};
        }}
        \path (-5.5,2) pic{N5={}};
        \path (0,2) pic{M3={}};
        \path (5.5,2) pic{triple={}};
    \end{tikzpicture}
    \caption{Lattices $N_5$, $M_3$, and critical triple structure. The critical triple is similar to $M_3$ but does not require meets $a_0\wedge a_1$, $b\wedge a_0$, or $b\wedge a_1$ to exist.} \label{fig:lattices}
\end{figure}

In a similar vein, to embed the critical triple, which is a structure resembling $M_3$ but with fewer restrictions (Figure~\ref{fig:lattices}), a c.e.\ degree only needs to change its mind more than $\omega$-times to succeed \cite{downey2007totally}. Therefore a degree's ability to embed certain lattices is highly dependent on the degree's fickleness (Downey and Greenberg use the phrase ``change mind frequently'', but we sometimes use the word ``fickle'' instead to be less verbose). This connection follows from our usual approach to constructing a c.e.\ degree that embeds a specific lattice, which is to prioritize the requirements for embeddability and satisfy them in turn. A more fickle degree is more likely to permit requirements to be satisfied, increasing the chances of a successful construction.\\

Motivated by how fickleness aids embeddability, Downey and Greenberg~\cite{downey2015transfinite} defined a hierarchy of c.e.\ degrees based on mind change frequency, in which degrees that are more fickle lie higher up the hierarchy. It turns out that c.e.\ degrees with ``reasonable'' fickleness are $\low_2$. Lying at the bottom level of the hierarchy are the most unfickle degrees, which turn out to be exactly the array computable ones~\cite{downey2015transfinite}.\\

Our work is based on this recent work of fickleness characterizing embeddability strength. Given that nonlowness is another measure of degree strength, we explore how lowness and nonlowness sits within the mind change hierarchy. We expect nonlow degrees to compute more degrees than low ones, but does that necessarily imply nonlows embed more types of lattices? This statement is believed to be false, for Downey and Greenberg~\cite{downey2015transfinite} claimed there are lows and nonlows on every level of the mind change hierarchy. But we are not aware of any formal proof of their claim, and provide one here.\\

Given an arbitrary ``reasonable'' ordinal $\alpha$, we construct low and nonlow c.e.\ sets that ``change their minds $\alpha$-times''. Our construction of the nonlow extends from our construction of a nonlow $\low_2$ c.e.\ set. It is not difficult to build nonlow $\low_2$'s, but we are not aware of a direct proof in the literature. We build such a set directly.\\

By distributing injury from nonlow requirements uniformly, our nonlow $\low_2$ set turns out to be array computable. Therefore, our construction also directly gives a nonlow array computable c.e.\ set, which are also sets whose proof of existence have always been indirect \footnote{First, construct an array computable, 1-topped, and incomplete c.e.\ set, then use the fact that every array computable and 1-topped degree is either incomplete or nonlow $\low_2$ \cite{downey1987t}. Alternatively, construct a nonlow c.e.\ degree that contains only a single $W$-degree and whose elements are all mitotic, then use the fact that every array noncomputable c.e.\ degree must contain a c.e.\ set that is non-$W$-mitotic \cite{downey1993array}.}.\\

After we build a nonlow array computable c.e.\ set $A$, given an arbitrary ``reasonable'' ordinal $\alpha$, we injure $A$ ``$\alpha$-more times'' but no more than that. Then $A$ will be a nonlow c.e.\ set that lies in the $\alpha$-level of the mind change hierarchy, as desired. To build a low set at the $\alpha$-level, we ``simplify'' the construction given by Downey and Greenberg (\cite{downey2015transfinite}, Lemma III.2.1), where the authors constructed a c.e.\ set at the $\alpha$-level. The authors used a tree in the construction, so the constructed set was not necessarily low. By tracking injury more carefully, we managed to avoid using a tree, ensuring the lowness of the constructed $\alpha$-level set.

%% file: 2-nonlow-ac.tex
\section{Nonlow, Array Computable}

We directly build a nonlow $\low_2$ c.e.\ set $A$ in Section~\ref{sec:nonlow-low2}, and show in Section~\ref{sec:array-computable} that $A$ is array computable. Our construction uses the ``usual'' tree framework to handle the infinite injury from the nonlow requirements. Refer to~\cite{soare1978recursively} Chapter 14 for priority constructions using trees. We introduce common terminologies in the following Section~\ref{sec:tree}.

\input{2.1-tree.tex}
\input{2.2-nonlow-low2-construction.tex}
\input{2.2-nonlow-low2-verfication.tex}
\input{2.3-array-computable.tex}

%% file: 2.1-tree.tex
\subsection{Tree constructions} \label{sec:tree}

We mainly follow terminology from \cite{downey2015transfinite} Chapter 1.1. Often in the construction we say that we pick a \emph{large} number as a use or a follower. Large refers to the first number greater than any number that has ever been seen or used up to the point of construction. Requirements are ordered $R_0<R_1<R_2<\ldots$, with smaller requirements having higher priority than larger ones.\\

\noindent Trees provide an intuitive framework for infinite injury constructions. Such constructions begin by fixing an ordering of the requirements $R_0<R_1<\ldots$, where each requirement has exactly one true outcome from an ordered and computable set $\texttt{outcomes} =\{o_0<o_1<o_2<\ldots\}$, with smaller outcomes having higher priority. The tree of construction is
\[\Lambda =\texttt{outcomes}^{<\omega}.\]

\noindent At each stage $s$ of the construction, we guess the outcomes of an initial segment of the requirements to get a node $\delta_s\in\Lambda$ of outcomes. Given $s\in\omega$ and $\delta\in\Lambda$, if $\delta\preceq\delta_s$, we say \emph{stage $s$ is an $\delta$-stage}, or \emph{$\delta$ is accessible at stage $s$}. Given $\delta,\delta'\in\Lambda$, we say \emph{$\delta$ lies to the left of $\delta'$}, written $\delta<_L\delta'$, to mean
\[\delta<_L\delta' \iff (\exists \eta\in\Lambda, o_0<o_1\in\texttt{outcomes})\; \left[\eta^\frown o_0\preceq \delta \text{ and } \eta^\frown o_1\preceq \delta'\right].\]

\noindent A node $\delta\in\Lambda$ is said to \emph{lie on the true path of outcomes} iff $\delta$ is accessed infinitely often but any node to the left of $\delta$ is only accessed finitely often. Formally,
\[\delta \text{ lies on true path of outcomes } \iff (\exists^\infty s)\; \left[\delta\preceq\delta_s \text{ and } (\forall \delta'<_L\delta)\; \left[\left|\{t: \delta'\preceq\delta_t \}\right|<\omega \right]\right].\]

\noindent Note that from this definition, the set of nodes that lie on the the true path must form a chain. We define the true path of outcomes $\delta_\omega$ as the union of this chain
\[\delta_\omega := \bigcup \{\delta\in\Lambda: \delta \text{ lies on the true path of outcomes}\},\]
and write $\delta\prec\delta_\omega$ to mean that $\delta$ lies on this true path. Often in tree constructions we want to show that $|\delta_\omega|=\omega$, i.e. the true path meets all requirements.

%% file: 2.2-nonlow-low2-construction.tex
\subsection{Nonlow, Low$_2$} \label{sec:nonlow-low2}

We give a direct construction of the following known fact:
\begin{theorem} \label{thm:nonlow-low2}
    There exists a nonlow $\low_2$ c.e.\ set $A$.
\end{theorem}

\underline{$N$-module ($\low_2$)}: To ensure $A$ is $\low_2$, $\emptyset''$ uniformly computes a series of computable sets $S_0,S_1,\ldots \equiv_T \emptyset$ such that $\forall e \in\omega$:
\[N_e: (\forall x)(\exists^\infty s\in S_e)\; [\Phi^A_e(x) [s]\downarrow] \implies \Phi^A_e \text{ total}.\]

\noindent To see how $N_e$ gives $\low_2$-ness, note that the left clause of $N_e$ is a $\Pi_2$-formula and therefore can be decided by $\emptyset''$. So $\emptyset''$ can determine if $\Phi^A_e$ is total, meaning that $\text{Tot}^A \leq_T \emptyset''$, making $A$ $\low_2$.\\

\noindent Let $N_e(x)$ denote the negative sub-requirement of $N_e$ that works to protect computation $\Phi^A_e(x)$. To reduce the number of variables used, if context is clear, we sometimes drop the subscript $e$ and write $N$ for $N_e$, and $N(x)$ for $N_e(x)$. Also, if $N=N_e$, we write $\Phi^A_N$ to mean $\Phi^A_e$.\\
 
\noindent To meet a single $N(x)$, wait for $\Phi^A_N\restriction x$ to converge for the first time. Then $N(x)$ will want to protect $\Phi^A_N(x)$ by ``initializing'' positive requirements $P$ that threaten $\Phi^A_N(x)$. But if $N(x)$ is overly protective, $P$-requirements may never get satisfied. As a compromise, each $N(x)$ is allocated a fixed \emph{quota} of requirements $P<N$ allowed to injure $N(x)$. The quota also limits the number of times a given $P<N$ can inflict injury. We provide details on quota allocation later. $N(x)$ ``initializes'' $P$ only if $P$ threatens to injure $\Phi^A_N(x)$ yet does not satisfy $N(x)$'s quota.\\

\underline{$P$-module (Nonlow)}: To make $A$ nonlow, fix a partial computable enumeration of the computable functions $\langle \langle \psi^e_s(x)\rangle_s \rangle_e$ and construct functional $\Gamma$ such that $\Gamma^A$ is c.e.\ in $A$ and satisfies positive requirements $\forall e\in\omega$:
\[P_e: (\forall x)\; \left[\lim_s \psi^e_s(x) \text{ exists} \right] \implies (\exists x)\; \left[\chi_{\Gamma^A}(x) \neq \lim_s \psi^e_s(x) \right].\]

\noindent To see how $P_e$ give nonlowness, note that $\Gamma^A$ is $\Sigma_1$ in $A$ and therefore computable by $K^A$. Also, the limits $\lim_s \psi^e_s$, when they exist, enumerate all the $\Delta^0_2$ functions by Shoenfield's limit lemma. Thus the $P_e$ requirements ensure no $\Delta^0_2$ function computes $K^A$, making $A$ nonlow.\\

\noindent To meet a single $P_e$, begin by picking a \emph{follower} $y$ for $P_e$. Wait for $\psi^e(y)[s]=0$. Then $P_e$ will want to pick a large \emph{use} $u$, and \emph{diagonalize $\Gamma^A$ out of $\psi^e$ at $y$} by declaring $\Gamma^{A\restriction u} (y)[s]\downarrow$. We say that $P_e$ wants to \emph{act via picking use}. But $P_e$ needs approval from all $N<P_e$ before acting so as not to injure these $N$-requirements too often. If some $N(x)$'s quota does not tolerate another action from $P_e$, then $P_e$ must wait for $\Phi^A_N(x)$ to stabilize first before acting. Quota for all $N(x)$'s is designed to be generous enough to allow $P_e$ to eventually act.\\

\noindent After acting, wait for $\psi^e(y)$ to change its mind to equal 1. Then $P_e$ will want to enumerate $u$ into $A$. The enumeration will diagonalize $\Gamma^A$ out of $\psi^e$ at $y$ because $\Gamma^{(A\cup\{u\})\restriction u} (y)[s]$ is not yet declared to converge and therefore diverges at stage $s$. We say that $P_e$ wants to \emph{act via enumerating use}. Like before, $P_e$ waits for approval from all $N<P_e$ to act. After acting, $P_e$ is considered to have no assigned use. So if $\psi^e(y)$ returns to 0 later, $y$ will need to be assigned a new use like before, and we repeat the earlier process of waiting for approval to act via picking use.\\

\noindent Observe that if the follower $y$ picked was such that $\lim_s \psi^e_s(y)$ does not exist, then $y$ will act infinitely often, enumerating infinite elements into $A$. Let $P_e(k)$ denote the sub-requirement of $P_e$ that works to enumerate $k$ elements into $A$ for $P_e$ since the beginning of the construction. To reduce the number of variables used, if context is clear, we drop the subscript $e$ and write $P$ for $P_e$, and $P(k)$ for $P_e(k)$. Also, if $P=P_e$, we write $\psi^P$ to mean $\psi^e$.\\

\underline{Overall Strategy:} The conflict between the $N$ and $P$ requirements is now clear: A given $P$ may enumerate infinite elements into $A$, injuring some $N$ infinitely often, yet each sub-requirement $N(x)$ of $N$ can tolerate only finite injury from all positive requirements. To distribute injury from $P$ across $N(x)$, we use a quota system where we assign each $N(x)$ a fixed quota $\texttt{quota}_N(x)$ of sub-requirements $P(k)$ allowed to injure $N(x)$. So if $P(k)$ lies in $\texttt{quota}_N(x)$, then $N(x)$ allows $P$ to act $k$ times. $\texttt{quota}_N(x)$ needs to be generous enough so that every $P$ gets infinite opportunities to act, yet also restrictive enough so that every $N(x)$ suffers only finite injury:
\begin{itemize}
    \item $(\forall N, P, k) (\exists \text{ cofinite } x)\; [P(k) \in \texttt{quota}_N(x)]$
    \item $(\forall N,x)\; [|\texttt{quota}_N(x)| <\omega]$.
\end{itemize}

\noindent To these ends, we set
\[\texttt{quota}_N(x) =\{P(k):\; (\exists k')\; [\langle P,k'\rangle < x \quad \&\quad 1\leq k\leq k']\}.\]

\noindent This quota is generous enough: Given $P(k)$, if $P$ wishes to act for the $k$-th time, then for all $x> \langle P,k\rangle$, $N(x)$ will allow $P$ to act. In other words, $N$ will allow $P$ to enumerate the $k$-th element the moment computations $\Phi^A_N(0), \ldots, \Phi^A_N(\langle P,k\rangle)$ have finalized.\\

\noindent At the same time, the quota is restrictive enough because $N(x)$ will only suffer injury in the scale of $|\texttt{quota}_N(x)|$, which is finite. Any $P$ with $P(1) \not\in \texttt{quota}_N(x)$ will always be initialized by $N(x)$ if $P$ threatens $\Phi^A_N(x)$, meaning there are only finitely many $P$ that can ever injure $N(x)$.\\

\noindent Consider when $N$ should allow $P$ to act via enumeration. If $P(1) \in\texttt{quota}_N(x)$, then $N(x)$ must never initialize $P$ if $P$ is to get enough chances to act. Therefore $N$ must always allow such $P$ to enumerate uses, even if $P$ has injured $N(x)$ $k$-times already and $P(k) \not\in \texttt{quota}_N(x)$.\\

\noindent Now consider when $N$ should allow $P$ to act via picking use. If $P$ wants to pick a new use, $N(x)$ will use the quota system to bound injury inflicted by $P$: If $P$ has acted $k$-times and $P(k+1)\in\texttt{quota}_N(x)$, then $N(x)$ will allow $P$ to pick a new use. But if $P(k+1)\not\in\texttt{quota}_N(x)$, $P$ must wait till computation $N(x)$ is ``correct'' first before being allowed to act. That is, $P$ must wait for the stage $s$ where for all $P'<P$, the use assigned to $P'$, if any, exceeds $\USE(\Phi^A_N(x)[s])$. By making $P$ wait, the excess injury on $N(x)$ after $P$'s $k$-th action can be bounded, even if $P$ was wrong about the ``correctness'' of computation $\Phi^A_N(x)$. The proof of boundedness uses a combinatorial argument, which we will show in Main Lemma 1.\ref{n:N-inf}.\\

\noindent Summarizing, the quota system distributes injury from $P$ across all $N(x)$ such that every $P$ only needs wait for finitely many $N(x)$ to finalize before being allowed to act, every $N(x)$ can only be injured by finitely many $P$, and each such $P$ can injure $N(x)$ only finitely often.\\

\underline{Tree Construction}: $P$-requirements can inflict either finite or infinite injury, so we say that $P$ has \emph{outcome} $\texttt{fin}$ or $\infty$. Given $N<P$, $N$ would play a different strategy depending on $P$'s outcome. If $P$ has outcome $\texttt{fin}$, then $N$ will wait for $P$ to finish acting before beginning to protect the $N$-computations. But if $P$ has outcome $\infty$, then $N$, knowing $P$ will enumerate larger and larger elements into $A$, will wait till the elements enumerated are too large to ever cause injury before protecting $\Phi^A_N$.\\

\noindent Similarly, $N$-requirements can injure $P$-requirements finitely or infinitely often, depending on whether $\Phi^A_N$ is non-total or total respectively. And $P$ changes strategy depending on the injury outcome of a higher priority $N<P$. If $N$ has outcome $\texttt{fin}$, then $P$ will wait for $N$ to stabilize before acting. And if $N$ has outcome $\infty$, then $P$ will wait for all sub-requirements $N(y)$ whose quota cannot tolerate another injury from $P$ to stabilize first before $P$ acts.\\

\noindent The framework of changing strategies based on the outcomes of higher priority requirements is often implemented using trees. Since our $P$ and $N$ requirements can have \emph{outcomes} from
\[\texttt{outcomes} :=\{\infty <\texttt{fin}\},\]
our tree of construction is
\[\Lambda :=\{\infty, \texttt{fin}\}^{<\omega}.\]
Prioritize the requirements $N_0 <P_0 <N_1 <P_1 <\ldots$. Then a node $\delta\in\Lambda$ works for $N_e$ if $|\delta|=2e$, and works for $P_e$ if $|\delta|=2e+1$. Given $\delta\in\Lambda$, if $\delta$ works for $N_e$, we write $\delta=\eta_e$. Often we drop the subscript $e$ and write $\eta$ in place of $\eta_e$ and $\Phi^A_\eta$ in place of $\Phi^A_e$. Also, we write $\eta(x)$ to refer to the sub-requirement of $\eta$ that works to protect $\Phi^A_\eta(x)$. Similarly, if $\delta$ works for $P_e$, we write $\delta=\rho_e$. We often drop the subscript $e$ and write $\rho$ for $\rho_e$ and $\psi^\rho$ for $\psi^e$. Also, we write $\rho(k)$ to mean the sub-requirement that works to get $\rho$ to act $k$ times since the beginning of the construction (as opposed to since the last time $\rho$ was initialized).\\

\noindent Let the current stage be $s$. Given a node $\delta\in\Lambda$, depending whether $\delta=\eta$ or $\delta=\rho$, we initialize and design $\delta$-strategy as follows:\\

\underline{Initialize $\eta$:} Do nothing.\\

\underline{Outcome of $\eta$:} Roughly speaking, we guess that $\eta$ has outcome $\infty$ iff more $\Phi^A_\eta$ computations have converged since the previous $\eta$-stage. However some of these computations may not be ``correct'' yet because they will subsequently be injured by $\rho\prec\eta$ with infinite outcome. Specifically, if $\rho^\frown\infty \preceq \eta$ and $\rho$'s use $u$ does not exceed $\USE(\Phi^A_\eta(x)[s] \downarrow)$, then $u$ will eventually be enumerated into $A$ because $\rho$ is expected to act infintely often. The enumeration injures computation $\Phi^A_\eta(x)$, so this computation is ``incorrect'' at stage $s$. Therefore when we guess the outcome of $\eta$, we count only the number of ``correct'' $\eta$-computations, and let $\eta$ have outcome $\infty$ iff this number has increased.\\

\noindent Formally, given $x\in\omega$ and comparable nodes $\delta,\delta'\in\Lambda$, we say \emph{computation $\Phi^A_\delta(x)[s]$ is $\delta'$-correct at stage $s$} if for all $\rho\preceq \delta'$ such that $\rho=\delta'$ or $\rho^\frown\infty \preceq \delta'$, if $\rho$ has an assigned use $u$ at stage $s$, then $u >\USE(\Phi^A_\delta(x)[s])$. Also, define the \emph{length of $\eta$ at stage $s$} as
\[l_s(\eta) :=\max \left\{x\leq s: (\forall y<x)\; \left[\Phi^A_\eta(y)[s]\downarrow \text{ and is } \eta\text{-correct} \right]\right\}.\]
We say \emph{$s$ is an $\eta$-expansionary stage} iff
\[l_s(\eta) >\max\{l_{s'}(\eta):\; s'<s \text{ and } s' \text{ is an } \eta\text{-stage}\},\]
and let $\eta$ have outcome $\infty$ iff $s$ is $\eta$-expansionary.\\

\underline{$(\eta^\frown\texttt{fin})$-strategy:} Do nothing.\\

\underline{$(\eta^\frown\infty)$-strategy:} Following the discussion in ``Overall Strategy'', given $x\in\omega$, allocate $\eta(x)$ a fixed quota $\texttt{quota}_\eta(x)$ to distribute injury from positive requirements across negative ones in a generous yet restrictive manner:
\begin{align} \label{eq:quota}
    \nonumber \texttt{quota}_\eta(x) =&\{\rho(k):\; \Phi^A_\eta(x) \text{ allows } \rho \text{ to act } k\text{-times}\}\\
    :=&\{\rho(k):\; (\exists k')\; \left[\langle \rho,k'\rangle < x\quad \&\quad 1\leq k\leq k'\right]\}.
\end{align}

\noindent Since $\texttt{quota}_\eta(x)$ does not depend on $\eta$, we drop the subscript $\eta$ and just write $\texttt{quota}(x)$. If $\rho(1) \in \texttt{quota}(x)$, we say \emph{$\rho$ lies in the quota of $x$}, and write $\rho\in\texttt{quota}(x)$. Define the \emph{quota for $\rho$ from $x$} as
\[\texttt{quotaFor}_x(\rho) :=\max \{k:\; \rho(k) \in \texttt{quota}(x)\}.\]
We say \emph{$\rho$ has exhausted its quota from $x$ at stage $s$} if $\rho$ has acted $k$-times at stage $s$ since the beginning of construction (as opposed to since the last time $\rho$ was initialized), but $\rho(k+1) \not\in \texttt{quota}(x)$. From earlier discussion, given $x\in\omega$, \emph{$\eta(x)$ allows $\rho\succeq \eta^\frown\infty$ to act via picking use at stage $s$} iff any of the following hold:
\begin{itemize}
    \item $\Phi^A_\eta(x)$ is $\rho$-correct at stage $s$
    \item $\rho$ has not yet exhausted its quota from $x$ at stage $s$
\end{itemize}
\emph{$\eta(x)$ allows $\rho$ to act via enumerating use at stage $s$} iff any of the following hold:
\begin{itemize}
    \item The use of $\rho$ at stage $s$ exceeds $\USE(\Phi^A_\eta(x)[s])$
    \item $\rho \in\texttt{quota}(x)$.
\end{itemize}
\emph{$\eta$ allows $\rho$ to act at stage $s$} iff for all $x<l_s(\eta)$, $\eta(x)$ allows $\rho$ to act at stage $s$.\\

\noindent Note that by $\eta(x)$'s strategy for allowing $\rho$ to act via enumeration, the only $\rho$ that can injure $\eta(x)$ are those that lie in the quota of $x$, and there are only finitely many such $\rho$. Also, given $\rho,\eta\in\Lambda$, only finitely many $\eta(x)$ can initialize $\rho$ because $\rho\in\texttt{quota}(x)$ for all large enough $x$. Therefore the quota system is both restrictive enough for $\eta(x)$ requirements and generous enough for $\rho$ requirements.\\

\underline{Initialize $\rho$:} Destroy the follower and use assigned to $\rho$, if any. So the next time $\rho$ is visited, $\rho$ is considered to have neither follower nor use.\\

\underline{Outcome of $\rho$:} If $\rho$ does not have a follower $y$ at this stage $s$, assign a large one. By our construction, if $\rho$ has an assigned use $u$ then $\Gamma^{A\restriction u}(y)[s] \downarrow$, and if $\rho$ does not have an assigned use, then $\Gamma^A(y)[s] \uparrow$. $\rho$ has outcome $\infty$ if $\psi^\rho(y)$ has changed its mind again such that $\Gamma^A(y)$ is no longer diagonalized out of $\psi^\rho(y)$. Formally, $\rho$'s outcome is $\infty$ iff any of the following hold:
\begin{enumerate}
    \item $\psi^\rho_s(y)=0$ and $\rho$ has no assigned use
    \item $\psi^\rho_s(y)=1$ and $\rho$ has an assigned use
\end{enumerate}
In the first case we say \emph{$\rho$ wants to act via picking use}, and in the second case we say \emph{$\rho$ wants to act via enumerating use}.\\

\underline{$(\rho^\frown\texttt{fin})$-strategy:} Let $y$ be the follower assigned to $\rho$. If $\psi^\rho_s(y)=0$ and $\rho$ has an assigned use $u$ at this stage, then \emph{$\rho$ diagonalizes $\Gamma^A_\rho$ out of $\psi^\rho$} by declaring $\Gamma^{A\restriction u}(y)[s]\downarrow$.\\

\underline{$(\rho^\frown\infty)$-strategy:} For $\rho$'s outcome to be $\infty$, $\rho$ must have wanted to act. If $\rho$ wants to act via picking use, then $\rho$ will want to pick a new large use $u$ for its follower $y$ and diagonalize $\Gamma^A_\rho$ out of $\psi^\rho$ by declaring $\Gamma^{A\restriction u} (y)[s] \downarrow$. If $\rho$ wants to act via enumerating use $u$, then $\rho$ will want to enumerate $u$ into $A$, so the next time $\rho$ is accessed, $\rho$ is considered to have no assigned use. \emph{$\rho$ is allowed to act} only if every $\eta^\frown\infty\preceq \rho$ allows $\rho$ to act at stage $s$. Refer to $(\eta^\frown\infty)$-strategy for what it means for $\eta$ to allow $\rho$ to act. If $\rho$ is not allowed to act via enumeration, then all $\delta\succeq \rho$ and $\delta\geq_L \rho$ will be initialized.\\

Playing the $\eta$ and $\rho$ strategies in a tree framework, we construct the nonlow $\low_2$ c.e.\ set $A$:

\begin{framed}
    \underline{Stage $s$:} Let $\delta_{s,0}$ be the empty node. From step $e=0$ to $s$: Determine the outcome $o\in\{\infty, \texttt{fin}\}$ of $\delta_{s,e}$. Set $\delta_{s,e+1} =\delta_{s,e}^\frown o$, and initialize all nodes to the right of $\delta_{s,e+1}$. If $\delta_{s,e+1} =\rho^\frown\texttt{fin}$, play the $(\rho^\frown\texttt{fin})$-strategy described above to diagonalize $\Gamma^A_\rho$ out of $\psi^\rho$.\\
    
    \noindent At the end of step $s$, we will get a node $\delta_s :=\delta_{s,s} \in\Lambda$ of length $s$. Some of the positive initial segments $\rho\preceq \delta_s$ may have outcome $\infty$, meaning that they want to act. Let
    \begin{align*}
        \Theta &=\{\rho\preceq\delta_s:\; \rho^\frown\infty \preceq \delta_s, \text{ and } \rho \text{ wants to pick use at stage } s \text{ and is allowed}\}\\
        &\cup\{\rho\preceq\delta_s:\; \rho^\frown\infty \preceq \delta_s, \text{ and } \rho \text{ wants to enumerate use at stage } s\}.
    \end{align*}
    
    \noindent If $\Theta$ is empty, go to the next stage. Otherwise, select one $\rho\in\Theta$. This $\rho$ will be called \emph{the selected node at stage $s$}, and is selected as:
    \[\rho =\underset{\rho'\in\Theta}{\mathrm{argmin}}\; \{\langle \rho',k\rangle:\; \rho' \text{ has been selected } k \text{-times before stage } s\}.\]
    Play the $(\rho^\frown\infty)$-strategy described above. Go to next stage.
\end{framed}

%% file: 2.2-nonlow-low2-verfication.tex
\begin{lemma} \label{lemma:delta-omega}
    $|\delta_\omega| =\omega$.
\end{lemma}
\begin{proof}
    Follows immediately from the facts that $|\delta_s|$ is increasing in $s$ and that $\Lambda$ is finitely branching.
\end{proof}

\begin{mainlemma1*}
    Given $n\in\omega$, let $\delta=\delta_\omega \restriction n \in \Lambda$.
    \begin{enumerate}
        \item \label{n:P-init} If $\delta=\rho$ then $\rho$ eventually stops being initialized. Thus $\rho$ has a final follower $y_\rho$.
        \item \label{n:P-fin} If $\delta=\rho$ and $\rho^\frown\texttt{fin} \prec \delta_\omega$, then if $\lim_s \psi^\rho_s(y_\rho)$ exists the limit will not equal $\chi_{\Gamma^A} (y_\rho)$.
        \item \label{n:P-inf} If $\delta=\rho$ and $\rho^\frown\infty \prec \delta_\omega$, then $\rho$ will act infinitely often. Thus $\lim_s \psi^\rho_s(y_\rho)$ does not exist.
        \item \label{n:N-init} If $\delta=\eta$, then $\eta$ eventually stops being initialized.
        \item \label{n:N-fin} If $\delta=\eta$ and $\eta^\frown\texttt{fin} \preceq \delta_\omega$, then $\Phi^A_\eta$ is not total.
        \item \label{n:N-inf} If $\delta=\eta$ and $\eta^\frown\infty \preceq \delta_\omega$, then $\Phi^A_\eta$ is total.
    \end{enumerate}
\end{mainlemma1*}

\vspace{2em}
\noindent We prove Main Lemma 1 after this immediate corollary:
\begin{corollary} \label{cor:A-nonlow-low2}
    $A$ is $\low_2$ nonlow.
\end{corollary}
\begin{proof}
    Let $e\in\omega$. We show that all $P_e$ and $N_e$ are satisfied. Let $\rho=\delta_\omega(2e+1)$, which works for $P_e$. If $\rho^\frown\texttt{fin} \prec\delta_\omega$, then from Lemmas~\ref{lemma:delta-omega} and Main Lemma 1.\ref{n:P-fin}, if $\lim_s \psi^e_s(y_\rho)$ exists, $P_e$ will be satisfied with witness $y_\rho$. If the limit does not exist, then $P_e$ will be vacuously true. On the other hand, if $\rho^\frown\infty \prec\delta_\omega$, then from Lemmas~\ref{lemma:delta-omega} and Main Lemma 1.\ref{n:P-inf}, $\lim_s \psi^e_s(y_\rho)$ cannot exist, so $P_e$ will be vacuously true.\\
    
    \noindent From Lemma~\ref{lemma:delta-omega} and Main Lemmas 1.\ref{n:N-fin} and 1.\ref{n:N-inf}, $\forall e\in\omega$,
    \[\Phi^A_e \text{ is total} \iff \delta_\omega(2e) =\infty.\]
    Then since $\delta_\omega \leq_T \emptyset''$, we get $\text{Tot}^A \leq_T \emptyset''$, so $A$ is $\low_2$.
\end{proof}

\vspace{2em}
We prove the \ref{n:N-inf} claims of Main Lemma 1 by simultaneous induction on $n$. In each claim for $\delta\in\Lambda$, always assume that we are working in $\delta$-stages after higher priority requirements or sub-requirements have ``stabilized''. These stages exist from induction hypothesis.\\

\noindent Formally, \emph{requirement $\delta\in\Lambda$ has stabilized at stage $s$} if no $\delta'\preceq \delta$ ever gets initialized at or after stage $s$. Such stages exist from induction hypothesis on Main Lemma 1.\ref{n:P-init} and 1.\ref{n:N-init}. \emph{Sub-requirement $\rho(k)$ has stabilized at stage $s$} if for all $\langle \rho',k'\rangle \leq \langle \rho,k\rangle$, if node $\rho'\in\Lambda$ is ever selected $k'$-times or less since the beginning of the construction, then these selections have already been made by stage $s$. \emph{Sub-requirement $\eta(x)$ has stabilized at stage $s$} if $\eta$ has stabilized at stage $s$, and computations $\Phi^A_\eta\restriction (x+1)$ are never injured again at or after stage $s$. Such stages exist from induction hypothesis on Main Lemma 1.\ref{n:N-inf}.

\begin{claim*}[Main Lemma 1.\ref{n:P-init}]
    If $\rho\prec \delta_\omega$, then $\rho$ eventually stops being initialized.
\end{claim*}
\begin{proof}
    Wait for all $\delta\prec \rho$ to stabilize. We can assume $\rho^\frown\infty \prec \delta_\omega$ otherwise $\rho$ will never be initialized again. From the $(\eta^\frown\infty)$-strategy, if $\eta(x)$ initializes $\rho$, then $\eta(x)$ must have satisfied $\eta^\frown\infty \preceq \rho$ and $\rho \not\in \texttt{quota}(x)$. There are only finitely many such $\eta(x)$, so wait for all of them to stabilize, which will eventually occur from induction hypothesis. Now these $\eta(x)$ will initialize $\rho$ when $\rho$ wants to act via enumeration but the use of $\rho$ does not exceed the use of $\eta(x)$. But after stabilizing, and after possibly one more initialization from $\eta(x)$, $\rho$ will always pick a use that exceeds $\USE(\eta(x))$. Then $\eta(x)$ will forever allow $\rho$ to enumerate use and therefore never initialize $\rho$ again.\\
\end{proof}

\begin{claim*}[Main Lemma 1.\ref{n:P-fin}]
    If $\rho^\frown\texttt{fin} \prec \delta_\omega$, then if $\lim_s \psi^\rho_s(y_\rho)$ exists the limit will not equal $\chi_{\Gamma^A} (y_\rho)$.
\end{claim*}
\begin{proof}
    From construction, $\rho$ will choose outcome $\infty$ the moment $\psi^\rho(y_\rho)$ agrees with $\chi_{\Gamma^A}(y_\rho)$.
\end{proof}

\begin{claim*}[Main Lemma 1.\ref{n:P-inf}]
    If $\rho^\frown\infty \prec \delta_\omega$, then $\rho$ will act infinitely often. Thus $\lim_s \psi^\rho_s(y_\rho)$ does not exist.
\end{claim*}
\begin{proof}
    Assume for contradiction that $\rho$ only gets to act $(k-1)$-times during the construction, which implies that $\rho$ was selected $\geq (k-1)$-times. Wait for $\rho$ and $\rho(k)$ to stabilize. Then $\rho$ will always have the highest priority to act.\\
    
    \noindent First consider the case where $\rho$'s $k$-th action is to pick a use. From the $(\eta^\frown\infty)$-strategy, if $\eta(x)$ disallows $\rho$ to pick use, then $\eta^\frown\infty \preceq \rho$ and $\rho$ has exhausted its quota from $x$. But since $\rho$ never acts more than $(k-1)$-times, there are only finitely many $\eta(x)$ satisfying these two conditions. Wait for all such $\eta(x)$ to stabilize. We show that $\eta(x)$ is eventually $\rho$-correct, so that $\rho$ will be allowed by $\eta(x)$ to pick use: Let $\rho'$ be such that $\eta(x)^\frown\infty \preceq \rho'^\frown\infty \preceq \rho$. We need to show that the uses of $\rho'$ eventually exceed $\USE(\eta(x))$. But this is true because $\rho'$ gets to act infinitely often from induction hypothesis.\\
    
    \noindent So it must be that $\rho$'s $k$-th action was to enumerate use. Now from construction, if $\rho$ has highest priority to act and $\rho$ wants to act via enumeration, then $\rho$ will be selected, even if some $\eta^\frown\infty \preceq \rho$ does not allow $\rho$ to enumerate use. And should such $\eta$ exist, then $\rho$ will be initialized. But $\rho$ has already stabilized, a contradiction.
\end{proof}

\begin{claim*}[Main Lemma 1.\ref{n:N-init}]
    If $\eta\prec \delta_\omega$, then $\eta$ eventually stops being initialized.
\end{claim*}
\begin{proof}
    Wait for all $\rho\prec\eta$ to stabilize. Then $\eta$ will never be initialized again.
\end{proof}

\begin{claim*}[Main Lemma 1.\ref{n:N-fin}]
    If $\eta^\frown\texttt{fin} \prec \delta_\omega$, then $\Phi^A_\eta$ is not total.
\end{claim*}
\begin{proof}
    Assume for contradiction $\Phi^A_\eta$ is total. We claim that given arbitrary $x$, computation $\Phi^A_\eta(x)$ is eventually $\eta$-correct. Equivalently, we are claiming that if $u$ is a use assigned to $\rho$ where $\rho^\frown\infty \prec \eta$, and $u <\USE(\Phi^A_\eta(x))$, then $u$ must eventually be enumerated into $A$. This claim is induction hypothesis with Main Lemma 1.\ref{n:P-inf}. Then $\eta^\frown\infty \prec \delta_\omega$, a contradiction.
\end{proof}

\begin{claim*}[Main Lemma 1.\ref{n:N-inf}]
    If $\eta^\frown\infty \prec \delta_\omega$, then $\Phi^A_\eta$ is total.
\end{claim*}

\noindent This claim is the heart of the argument on why the construction works. Always assume we are working in the $(\eta^\frown\infty)$-stages after $\eta$ stabilizes. The proof idea is that given $x\in\omega$, the moment $\Phi^A_\eta \restriction (x+1) \downarrow$, the only $\rho\in\Lambda$ that can injure $\eta(x)$ are those that lie in the quota of $x$. But there are only finitely many such $\rho$, and we will show that each of them will eventually stop injuring $\eta(x)$:\\

\noindent The first $\rho$ to stop injuring $\eta(x)$ are those that know the outcomes of higher priority injurious nodes. That is, $\rho$ can injure $\eta(x)$, but no $\rho'\succeq\rho^\frown\infty$ can injure $\eta(x)$ again. We call such nodes \emph{edge nodes}, for they can be seen as the outer most layer of nodes that can injure $\eta$. Wait for edge nodes $\rho$ to exhaust their quota from $x$. Then $\rho$ can act only when $\eta(x)$ is $\rho$-correct; that is, when $\rho$ believes the action will not injure $\eta(x)$. But having no extension that can injure $\eta(x)$, $\rho$ has complete knowledge of the outcomes of all relevant injurious nodes. In particular, when $\rho$ believes an action will never injure $\eta(x)$, that action will indeed never injure $\eta(x)$.\\

\noindent Therefore after exhausting quota, edge nodes can never injure $\eta(x)$, and can be removed from the set of nodes that are injurious to $\eta(x)$. Peeling away this outer layer of nonthreatening nodes exposes an inner layer of nodes, which will become the new edge nodes. This inner layer nodes will also eventually exhaust quota and stop injuring $\eta(x)$. Peeling away the inner layer nodes exposes an even inner layer of new edge nodes. Repeating this process, since only finitely many nodes can injure $\eta(x)$, eventually all layers of these nodes will be peeled away until no nodes are left to injure $\eta(x)$, completing the proof.\\

Formally, let $x\in\omega$ and $\eta^\frown\infty \prec \delta_\omega$, and we will show that $\Phi^A_\eta(x) \downarrow$. Work only in $(\eta^\frown\infty)$-stages after $\eta$ stabilizes. The following fact follows directly from the $(\eta^\frown\infty)$-strategy of the construction:

\begin{fact} \label{fact:inj-set}
    For all $s \in\omega$ and $\eta\in\Lambda$,
    \[\{\rho: \rho \text{ will injure } \eta(x) \text{ at an } (\eta^\frown\infty)\text{-stage}\} \subseteq \{\rho:\; \rho \in \texttt{quota}(x)\},\]
    Note that the larger set is finite, computable, and independent from $\eta$.
\end{fact}

\begin{definition} \label{def:edge}
    We call $\rho$ an edge node for $\eta(x)$ at stage $s$ if no $\rho'\succeq \rho^\frown\infty$ injures $\eta(x)$ at or after stage $s$.
\end{definition}

\begin{claim} \label{claim:edge}
    If $\rho\in\texttt{quota}(x)$ is an edge node for $\rho$ at stage $s$, the injury that $\rho$ can inflict on $\eta(x)$ after stage $s$ is bounded by $\texttt{quotaFor}_x(\rho)+1$.
\end{claim}
\begin{proof}
    Wait for $\rho$ to exhaust its quota from $x$ after stage $s$, inflicting no more than $\texttt{quotaFor}_x(\rho)$ injury on $\eta(x)$. We show that $\rho$ cannot injure $\eta(x)$ more than once: Then if $\rho$ has an assigned use, wait for that use to be enumerated, possibly injuring $\eta(x)$ one last time. If the use is never enumerated, then $\rho$ can never injure $\eta(x)$, and we are done. After enumerating use, assume that $\rho$ eventually wants to pick a new use, otherwise we are again done. Then since $\rho$ has exhausted its quota from $x$, the $(\eta^\frown\infty)$-strategy will only allow $\rho$ to pick a new use when $\eta(x)$ is $\rho$-correct.\\
    
    \noindent Assume for contradiction that $\rho$ injures $\eta(x)$ at stage $s_2$, via a use that was picked at stage $s_0<s_2$. That means between picking that use and injuring $\eta(x)$, there must have been a first $\rho'\succeq \eta^\frown\infty$ that injured $\eta(x)$, say at stage $s_1>s_0$, $s_1<s_2$. Now $\rho'\not\preceq \rho$ since $\eta(x)$ was $\rho$-correct when $\rho$ picked its new use. Yet $\rho'\not\succeq \rho$ because no $\rho'\succeq \rho$ can injure $\eta(x)$, by choice of $\rho$ being an edge node. Also, $\rho'\not<_L \rho^\frown\infty$ otherwise $\rho$ would be initialized at stage $s_1$ and cannot injure $\eta(x)$ later with the use that was picked at stage $s_0$. Finally, $\rho'\not>_L\rho$ because $\rho$ would have initialized $\rho'$ at stage $s_0$, implying that the use of $\rho'$ at stage $s_1$ must exceed the use of $\rho$ at stage $s_0$. So $\rho'$ cannot exist.
\end{proof}

\begin{definition}
    The edge layer of $\rho$ with respect to $x$ is defined as
    \[\max_{\substack{\rho'\succeq \rho^\frown\infty\\ \rho'\in \texttt{quota}(x)}}(|\{\rho'': \rho^\frown\infty \preceq \rho'' \preceq \rho'\}|).\]
\end{definition}

\noindent So the the outer most layer of nodes that lie in $\texttt{quota}(x)$ have edge layer 0, and inner layer nodes have higher edge layer number. In particular, if $\rho'\succ\rho$, then $\rho'$ has smaller edge layer than $\rho$.

\begin{proof}
    (Of Main Lemma 1.\ref{n:N-inf}): From Fact~\ref{fact:inj-set}, only nodes that lie in $\texttt{quota}(x)$ can injure $\eta(x)$. Let $\rho\in\texttt{quota}(x)$, and $r$ be the edge layer of $\rho$ with respect to $x$. We prove by induction on $r$ that $\rho$ eventually stops injuring $\eta(x)$. The base case is Claim~\ref{claim:edge}. From induction hypothesis, all nodes $\rho'\in\texttt{quota}(x)$ with edge layer $<r$ will eventually stop injuring $\eta(x)$. When that happens $\rho$ becomes an edge node for $\eta(x)$, and will also stop injuring $\eta(x)$ from Claim~\ref{claim:edge}. Since $|\texttt{quota}(x)|$ is finite, eventually all nodes stop injuring $\eta(x)$.
\end{proof}

%% file: 2.3-array-computable.tex
\subsection{Nonlow, Array Computable} \label{sec:array-computable}

We show in this section that the nonlow $\low_2$ c.e.\ set constructed in Theorem~\ref{thm:nonlow-low2} is array computable. Array computability was introduced by Downey, Jockusch, and Stob~\cite{downey1990array}, after observing how sufficient fickleness was enough to satisfy many commonly encountered embeddability requirements. Array computable degrees are the least fickle c.e.\ degrees, and can be characterized as follows:

\begin{definition}
    A c.e.\ degree $\bm{d}$ is array computable~\cite{downey1996array} iff there exists a computable $m: \omega\to\omega$ such that every $A\in \bm{d}$ has an $m$-bounded computable approximation. That is, $A$ has a uniformly computable approximation $a_s(x) \equiv_T \emptyset$ such that $\forall x\in \omega$,
    \begin{itemize}
        \item $A(x) =\lim_s a_s(x)$, and
        \item $|\{s\in\omega:\; a_s(x) \neq a_{s+1}(x)\}| \leq m(x)$.
    \end{itemize}
\end{definition}

\noindent Array computability turns out to characterize the c.e.\ degrees with a strong minimal cover~\cite{ishmukhametov1999weak}, showing again the connection between fickleness and lattice structure. The set just constructed has the least possible fickleness:

\begin{theorem} \label{thm:nonlow-ac}
    The nonlow $\low_2$ c.e.\ $A$ constructed in Theorem~\ref{thm:nonlow-low2} is array computable.
\end{theorem}

We use notations from the proof of Theorem~\ref{thm:nonlow-low2}. Let $\eta^\frown\infty \prec \delta_\omega$. Work only in the $(\eta^\frown\infty)$-stages after $\eta$ has stabilized. Define
\begin{align}
    \label{eq:S} S_\eta &:=\{s\in\omega:\; s \text{ is } (\eta^\frown\infty)\text{-stage after } \eta \text{ has stabilized}\},\\
    \label{eq:t} t_\eta(x) &:=(\mu\; s\in S_\eta)\; \left[\Phi^A_\eta \restriction (x+1)[s] \downarrow \right],\\
    \label{eq:g} a_{\eta,s}(x) &:=\Phi^A_\eta(x)[s'], \quad \text{ where } s' =\mu(s'\in S_\eta)\; [s' \geq \max(s,t_\eta(x))].
\end{align}

\noindent Note that $a_{\eta,s}(x)$ is the canonical computable approximation of $\Phi^A_\eta(x)$ that we get from our construction. Also note that for a fixed $\eta$, $S_\eta$ and $t_\eta(x)$ are computable, but not uniformly computable. We need to define a computable $m:\omega\to\omega$ such that for every $\eta^\frown\infty \prec \delta_\omega$ and $x\in\omega$,
\[|\{s\in S_\eta, s\geq t_\eta(x):\; a_{\eta,s}(x) \neq a_{\eta,s+1}(x)\}| \leq m(x).\]

\noindent Equivalently, we want to computably bound the injury on $\eta(x)$ at the $S_\eta$-stages after stage $t_\eta(x)$, and also show that this bound is independent from $\eta$. Our proof outline follows the proof of Main Lemma 1.\ref{n:N-inf}, where we bounded injury on $\eta(x)$. Independence from $\eta$ follows directly from the fact that $\eta$'s quota for $x$ $\texttt{quota}_\eta(x)$ (Eq.~(\ref{eq:quota})) does not depend on $\eta$. Therefore our proof of computable boundedness for a given $\eta$ will also hold for arbitrary $\eta$.\\

\noindent Now we show that the injury bound in Main Lemma 1.\ref{n:N-inf} is computable as a function of $x$. Fact~\ref{fact:inj-set} gave a computable and finite bound on the set of nodes that can injure $\eta(x)$. We computably bound the injury from each $\rho$ in the set to computably bound the total injury on $\eta(x)$. We show that after exhausting quota, any injury inflicted by $\rho$ must be ``triggered'' by injury from some $\rho'\neq\rho$ of smaller edge layer than $\rho$:\\

\noindent Wait for $\rho$ to exhaust quota from $x$. Then $\rho$ will only act when $\eta(x)$ is $\rho$-correct. But because $\rho$ may not yet be an edge node, $\rho$ may not see the outcomes of injurious nodes extending $\rho^\frown\infty$, and may misjudge the correctness of $\eta(x)$. For example, $\rho$ may pick a use at a $\rho$-correct stage $s$, thinking that computation $\eta(x)$ has finalized. Unknowing to $\rho$, some $\rho'$ with use $u'<\USE(\Phi^A_\rho(x)[s])$ will injure $\eta(x)$ later. The restored use of $\eta(x)$ exceeds $u$, so $\rho$ is triggered to unexpectedly injure $\eta(x)$ when $\rho$ enumerates $u$ into $A$. But $\rho'$ cannot have been an initial segment of $\rho$, because otherwise $\rho$ would have seen the outcome of $\rho'$, contradicting $s$ being a $\rho$-correct stage. In fact, $\rho'$ also cannot lie to the left or right of $\rho$ because the tree strategy makes these nodes irrelevant to the argument. So $\rho'$ must extend $\rho^\frown\infty$ as nodes, implying that $\rho'$ has smaller edge layer than $\rho$.\\

\noindent Summarizing, after exhausting quota, whenever $\rho$ injures $\eta(x)$, $\rho$ must be ``triggered'' by another node of smaller edge layer. Therefore the number of times $\rho$ injures $\eta(x)$ cannot exceed the total injury from smaller edge layers. Then by strong induction on edge layer number, we can computably bound the injury from any node in $\texttt{quota}(x)$. Summing these bounds gives a computable bound of total injury on $\eta(x)$. Formally:\\

\begin{definition} \label{def:trigger}
    Assume $\rho\in\texttt{quota}(x)$ injures $\eta(x)$ at stage $s$ by enumerating its use $u$ into $A$, and use $u$ was picked by $\rho$ at some $\rho$-stage $s'\leq s$ after $\rho$ has exhausted quota from $x$. Then $\USE(\eta(x)[s'])<u$, so between stages $s'$ and $s$, there must have been a first $\delta\neq\rho$ that injured $\eta(x)$, causing the use of $\eta(x)$ to eventually exceed $u$. We say that $\delta$ is the node that triggered $\rho$ to injure $\eta(x)$ at stage $s$.
\end{definition}

\begin{claim} \label{claim:trigger}
    If $\delta$ is the node that triggered $\rho$ to injure $\eta(x)$ at stage $s$, then $\delta\succeq \rho^\frown\infty$. In particular, if $\delta=\rho'$, then $\rho'$ has smaller edge layer than $\rho$.
\end{claim}
\begin{proof}
    Let $s_0<s$ be the $\rho$-stage after $\rho$ exhausted quota and where $\rho$ picked the use which injured $\eta(x)$ at stage $s$. At stage $s_0$, $\eta(x)$ was $\rho$-correct, implying $u>\USE(\Phi^A_\eta(x)[s_0])$. Let $s_1>s_0$ be the $\delta$-stage where $\delta$ enumerated its use $u'$ into $A$, injuring $\eta(x)$ and triggering the future injury from $\rho$. Note that $u'<\USE(\Phi^A_\eta(x)[s_0])$ and $u'<u$.\\
    
    \noindent Now $\delta\not<_L \rho^\frown\infty$, otherwise $\delta$ will initialize $\rho$ at stage $s_1$, then $\rho$ cannot have injured $\eta(x)$ at stage $s_2$ with use $u$. Also, $\rho \not<_L \delta$, otherwise $\rho$ would have initialized $\delta$ at stage $s_0$, contradicting $u'<u$. Finally, $\delta\not\preceq\rho$, otherwise $\eta(x)$ would not be $\rho$-correct at stage $s_0$ since $u'<\USE(\Phi^A_\eta(x[s_0])$. So $\delta$ must extend $\rho^\frown\infty$ as nodes.\\
\end{proof}

\begin{definition}
    Define the injury power of $\rho$ on $\eta(x)$ as
    \[\texttt{InjPow}_\eta(x,\rho) := \text{Number of times } \rho \text{ injures } \eta(x) \text{ at the } S_\eta\text{-stages after } t_\eta(x).\]
\end{definition}

\begin{claim} \label{claim:edge-inductive}
    If $\rho\in\texttt{quota}(x)$ has edge layer $r$, then
    \begin{align} \label{eq:injpow}
        \texttt{InjPow}_\eta(x,\rho) \leq \texttt{quotaFor}_x(\rho) +\sum_{\substack{\rho'\in \texttt{quota}(x)\\ \rho' \text{ has edge layer } <r}}\; \texttt{InjPow}_\eta(x,\rho').
    \end{align}
\end{claim}
\begin{proof}
    The $\texttt{quotaFor}_x(\rho)$ term comes from $\rho$ exhausting quota from $x$, and the summation term comes from Claim~\ref{claim:trigger}.
\end{proof}

\begin{proof}
    (Of Theorem~\ref{thm:nonlow-ac}): From Fact~\ref{fact:inj-set}, working only in the $(\eta^\frown\infty)$-stages after stage $t_\eta(x)$, the number of times $\eta(x)$ gets injured equals $\sum_{\rho\in\texttt{quota}(x)} \texttt{InjPow}_\eta(x,\rho)$, where we can bound each $\texttt{InjPow}_\eta(x,\rho)$ computably using the recursive relation Eq.~(\ref{eq:injpow}). Together with the fact that $\texttt{quota}(x)$ is a computable set, the total injury on $\eta(x)$ will be computable and independent from $\eta$, which completes the proof.\\
    
    \noindent To give a concrete bound of the injury on $\eta(x)$, we assume that the pairing function used for the definition of $\texttt{quota}(x)$ is such that $\max \{\max(|\rho|, k):\; \langle \rho,k\rangle \in \texttt{quota}(x)\} <x$. Then we can show that
    \begin{equation} \label{eq:injpow-tot}
        \text{Number of times } \eta(x) \text{ gets injured} \leq x^24^{x^2}.
    \end{equation}
    
    \noindent To prove the above, we first apply induction on edge layer $r$ of $\rho\in\texttt{quota}(x)$ with Eq.~(\ref{eq:injpow}) to show that $\texttt{InjPow}_\eta(x,\rho) \leq x(r+1)4^{(r+1)^2}$: The claim holds trivially if $r=0$, so assume $r\geq1$. By assumption on the pairing function, for all $\rho'$, $\texttt{quotaFor}_x(\rho') <x$, so Eq.~(\ref{eq:injpow}) becomes
    \begin{align*}
        \texttt{InjPow}_\eta(x,\rho) \leq &x +\sum_{\substack{\rho'\in \texttt{quota}(x)\\ \rho' \text{ has edge layer } <r}}\; \texttt{InjPow}_\eta(x,\rho')\\
        \leq &x +|\{\rho'\succ\rho:\; \rho' \text{ edge layer } <r\}| \cdot \max_{\rho'' \text{ edge layer } r-1} \texttt{InjPow}_\eta(x,\rho'')\\
        \leq &x +2^r \cdot \max_{\rho'' \text{ edge layer } r-1} \texttt{InjPow}_\eta(x,\rho'')\\
        \leq &x +2^r \cdot \left(xr4^{r^2}\right) &(\text{induction hypothesis})\\
        \leq &x2^{2r^2+r} +2^r \cdot \left(xr4^{r^2}\right) &(\because\; r\geq1)\\
        = &x(r+1)2^{2(r+1)^2-3r-2}\\
        \leq &x(r+1)2^{2(r+1)^2}\\
        = &x(r+1)4^{(r+1)^2},\\
    \end{align*}
    which completes the claim. Then by our assumption on the pairing function, 
    \begin{align*}
        \text{Number of times } \eta(x) \text{ gets injured} = &\sum_{\rho\in\texttt{quota}(x)} \texttt{InjPow}_\eta(x,\rho)\\
        \leq & \texttt{InjPow}_\eta(x,\rho), &\text{where } \rho \text{ has edge layer } x\\
        \leq &x(x+1)4^{(x+1)^2}\\
        \leq &(x+1)^24^{(x+1)^2}.
    \end{align*}
\end{proof}

%% file: 3-low-alpha.tex
\section{Low, Totally $\alpha$-c.a.}

Given ``reasonable'' ordinal $\alpha$, we build a low c.e.\ set $A$ that ``changes its mind $\alpha$-times''. Before constructing $A$, we formalize in the following Section~\ref{sec:alpha} what it means for $A$ to make $\alpha$-mind-changes, and what it means for $\alpha$ to be reasonable.

\input{3.1-alpha.tex}
\input{3.2-low.tex}

%% file: 3.1-alpha.tex
\subsection{Totally $\alpha$-c.a.} \label{sec:alpha}
By Shoenfield's Limit Lemma, every $\Delta^0_2$ set $A$ has a \emph{computable-approximation} $\langle a_s(x) \rangle_s$, which is a uniformly computable sequence such that for all $x\in\omega$,
\[A(x) =\lim_s a_s(x).\]

\noindent To measure how fickle $A$ is, we want to keep track of how often $a_-(x)$ changes its mind. This idea was formalized as follows:

\begin{definition}[\cite{downey2015transfinite} II.D1.1] ~\label{def:comp-approx}
    Let $\mathcal{R} =(R, <_\mathcal{R})$ be a computable well-ordering (both $R$ and $<_\mathcal{R}$ are computable), and $A$ be a $\Delta^0_2$ function. Then an $\mathcal{R}$-computable-approximation of $A$ is a computable-approximation $\langle a_s\rangle$ of $A$, together with a uniformly computable sequence $\langle m_s \rangle_s$ of functions $m_-(x):\omega\to R$ such that for all $x$ and $s$:
    \begin{itemize}
        \item $m_{s+1}(x) \leq_\mathcal{R} m_s(x)$
        \item if $a_{s+1}(x) \neq a_s(x)$, then $m_{s+1}(x) <_\mathcal{R} m_s(x)$.
    \end{itemize}
    
    \noindent The sequence $\langle m_s\rangle$ is called the \emph{mind change function of $A$}.
\end{definition}

\noindent It is tempting to define $A$ as changing its mind $\alpha$-times if $A$ has an $\mathcal{R}$-computable-approximation of \emph{order type} $\otp(\mathcal{R}) =\alpha$. But we want a notion of mind changes that gives a non-trivial hierarchy, and such that higher levels of the hierarchy have stronger lattice-embedding abilities. If $\mathcal{R}$ was only required to be computable, then the hierarchy would collapse to the $\omega$-level (Ershov): Given $A\in\Delta^0_2$ and letting $\langle \langle a_s(x)\rangle_s \rangle_x$ be a computable approximation of $A$, the computable well ordering $\mathcal{R}$ with $R =\{\langle x,s\rangle:\; a_s(x) \neq a_{s+1}(x)\}$ and ordering $\langle x,s\rangle <_\mathcal{R}\langle x',s'\rangle \iff [x<x' \text{ or } (x=x'\; \&\; s>s')]$ has order type $\omega$, witnessing $A$ being at the $\omega$ level.\\

\noindent Even if we require $\mathcal{R}$ to satisfy the additional properties of having its set of limit points $L(\mathcal{R})$ and its successor function $S_{\mathcal{R}}:R\to R$ to also be computable, the hierarchy would still collapse, though to the $\omega^2$-level (Ershov): Let $\mathcal{R}$ be the computable well ordering with $\otp(\mathcal{R}) =\omega$ defined above. Then $\omega \cdot \mathcal{R}$ will be be a computable well ordering of order type $\omega^2$. Also, $L(\omega\cdot \mathcal{R})$ will be a computable set, because $\langle n,z\rangle \in L(\omega\cdot \mathcal{R})$ iff $n\neq0$ and $z$ is the $\mathcal{R}$-smallest element. Finally, $S_{\omega\cdot\mathcal{R}}(\langle n,z\rangle)$ is a computable function, because $S_{\omega\cdot\mathcal{R}}(\langle n,z\rangle) =\langle n+1,z\rangle$. Thus $\omega\cdot\mathcal{R}$ witnesses $A$ being at the $\omega^2$ level. Therefore a meaningful definition of mind changes calls for $\mathcal{R}$ to satisfy further properties.\\

\noindent In lattice embeddability constructions, we often need the notion mind changes to be independent from the $\mathcal{R}$ used. Say we want to diagonalize out of the sets that change their minds less than $\alpha$-times. To effectively enumerate these less fickle sets, we fix some $\mathcal{R}$ with $\otp(\mathcal{R})=\alpha$ and enumerate all sets with an $\mathcal{R}$-computable-approximation. However, there might be some $A$ that also changes its mind less than $\alpha$-times, but via a different computable well-ordering $\mathcal{R}'$. Then our enumeration will not include $A$ if there is no procedure to get to $\mathcal{R}$ from $\mathcal{R}'$ effectively. Therefore $\mathcal{R}$ must be such that if $A$ changes its mind $\alpha$-times via $\mathcal{R}$ and also $\alpha'$-times via $\mathcal{R}'$ and $\alpha>\alpha'$, then $A$ changes its mind $\alpha'$-times via $\mathcal{R}\restriction \alpha'$.\\

\noindent Another crucial property for $\mathcal{R}$ to have is sufficient informativeness. Given $r\in R$, we want to have a rough idea of how large the ordinal with notation $r$ is. For example, say we want to use a sufficiently fickle $A$ to build another set satisfying certain requirements. At each stage of the construction, we check the number of mind changes that $A$ has left, and decide a strategy based on that number $r\in R$. If $A$ has more than $\omega^\omega$-mind changes left for example, we might play a different strategy than if $A$ has only $\omega$ mind changes left. Therefore $\mathcal{R}$ needs to contain enough information to compare ordinal notations meaningfully.\\

\noindent To these ends, it is enough for $\mathcal{R}$ to be \emph{canonical}:

\begin{definition}[\cite{downey2015transfinite} II.D2.2]
    Let $\mathcal{R}$ be a computable well-ordering with $\otp(\mathcal{R}) =\alpha$. Then $\text{nf}_\mathcal{R}: \omega\rightarrow (\omega^2)^{<\omega}$ denotes the function that takes each ordinal below $\alpha$ to its Cantor-normal form, i.e. $\forall z\in R$
    \[\nf_\mathcal{R}(z) = \langle \langle z_0, n_0 \rangle, \ldots, \langle z_i, n_i\rangle \rangle,\]
    where $z_j\in R$, $|z_0| > \ldots > |z_i|$, $n_j \in \omega-\{0\}$, and
    \[|z| = \omega^{|z_0|}\cdot n_0 + \ldots + \omega^{|z_i|} \cdot n_i\]
    gives the unique Cantor-normal form of $|z|$. If $\nf_\mathcal{R}$ is computable, we say that $\mathcal{R}$ is canonical.
\end{definition}

\noindent We are now ready to formalize mind change frequency of a set $A$ and of a c.e.\ degree:
\begin{definition}[\cite{downey2015transfinite} II.D2.9] \label{def:alpha-sets}
    $A\in\Delta^0_2$ is $\alpha$-c.a.\ if $A$ has an $\mathcal{R}$-computable-approximation for some canonical $\mathcal{R}$ with $\otp(\mathcal{R}) =\alpha$. Informally, if $A$ is $\alpha$-c.a., we say that set $A$ changes its mind $\alpha$-times.
\end{definition}

\begin{definition}[\cite{downey2015transfinite} III.D1, III.L1.1] \label{def:tot-alpha}
    Degree $\bm{d}\in\Delta^0_2$ is totally $\alpha$-c.a.\ if every $A\leq_T \bm{d}$ is $\alpha$-c.a.. Informally, we say that degree $\bm{d}$ changes its mind $\alpha$-times.
\end{definition}

\noindent Note that when $A\in\Delta^0_2$ changes its mind $\alpha$-times, even when $\alpha$ is infinite, $A(x)$ actually only changes its mind finitely often from the definition of computable approximation (Definition~\ref{def:comp-approx}), together the fact that there is no infinite descending sequence of ordinals. Therefore when we construct a set $A$ with $\geq\omega$ mind changes, the mind change requirement will inflict only finite injury.\\

\noindent By requiring that $\mathcal{R}$ be canonical, the mind change hierarchy obtained is $\low_2$ (\cite{downey2015transfinite} III.T1.2), and gives meaningful lattice embeddability results - the degrees above the $\omega$-level are exactly those below which one can embed the critical-triple \cite{downey2007totally}, and the degrees above the $\omega^\omega$-level are exactly those below which one can embed $M_3$ \cite{downey2015transfinite}. Also, if we focus on the levels below
\[\epsilon_0 := \sup\{\omega, \omega^\omega, \omega^{\omega^\omega}, \ldots\},\]
then a set's fickleness will not depend on the $\mathcal{R}$ picked (\cite{downey2015transfinite} II.P2.3, P2.8). Looking at these lower level degrees is enough for now because all the embeddability results we have discussed lie below $\epsilon_0$. Finally, as desired, the mind change hierarchy is non-trivial:
\begin{fact}[\cite{downey2015transfinite} III.L2.1] \label{fact:alpha-not-beta}
    Let $\alpha\leq \epsilon_0$, and let $A$ be a totally $\alpha$-c.a.. If $\alpha$ is a power of $\omega$, then there exists a c.e.\ set $A$ that is totally $\alpha$-c.a.\ but not totally $\beta$-c.a.\ for any $\beta<\alpha$.
\end{fact}

\noindent The converse of the above fact is also true:
\begin{fact}[\cite{downey2015transfinite} III.L2.2] \label{fact:alpha-beta}
    Let $\alpha\leq \epsilon_0$, and let $A$ be a totally $\alpha$-c.a.. If $\alpha$ is not a power of $\omega$, then $A$ is totally $\beta$-c.a.\ for some $\beta<\alpha$.
\end{fact}

\noindent The key ingredient that makes Fact~\ref{fact:alpha-not-beta} hold is:
\begin{fact} \label{fact:sum-power-omega}   
    Let $\alpha<\epsilon_0$. Then $\alpha$ is closed under ordinal addition if and only if $\alpha$ is a power of $\omega$.
\end{fact}

\noindent To see how $\alpha$ being closed under addition is crucial for constructing sets $A$ at the $\alpha$ level, consider the requirements that $A$ will need to satisfy. To avoid being totally $\beta$-c.a.\ for any $\beta<\alpha$, $A$ needs to compute some total $\Delta^A$ that diagonalizes out of all the $\beta$-c.a.\ functions. This positive requirement calls for an uniform enumeration of all $\beta$-c.a.\ functions, which we can get from the following fact:
\begin{fact}[\cite{downey2015transfinite} II.L2.6] \label{fact:enum}
    Let $\alpha\leq \epsilon_0$. There is a uniform enumeration of the $\beta$-c.a.\ functions for all $\beta<\alpha$. Formally, there is a canonical $\mathcal{R}$ of order-type $\alpha$, a computable function $g(e):\omega \to R$ known as the \emph{bounding function}, and a uniformly computable series of functions $\langle \langle f^e_s(x), m^e_s(x)\rangle_s \rangle_e$ known as the \emph{computable-approximation functions}, such that for all $e\in\omega$, $\langle f^e_s(x), m^e_s(x)\rangle_s$ is a $g(e)$-computable-approximation of $f^e(x) :=\lim_s f^e_s(x)$, and $\langle f^e\rangle_{e\in\omega}$ enumerates the $\beta$-c.a.\ functions for all $\beta<\alpha$.
\end{fact}

\noindent Now when diagonalizing $\Delta^A$ out of a $\beta$-c.a.\ function $f^e$, $\beta$ elements \footnote{The number of elements is finite, but the elements can only be indexed canonically by a possibly infinite ordinal $\beta$.} may be enumerated into $A$. At the same time, for $A$ to be totally $\alpha$-c.a., if $\Phi^A$ is total, then for all $x\in\omega$, $\Phi^A(x)$ cannot change its mind more than $\alpha$-times. So $\Phi^A(x)$ must tolerate $\beta$-injury from each of its higher priority positive requirements. There are only finitely many such requirements, so the total injury on $\Phi^A(x)$ is a finite sum of ordinals $\beta<\alpha$. This sum can be bounded below $\alpha$ if and only if $\alpha$ is closed under ordinal addition. Therefore Fact~\ref{fact:sum-power-omega} is crucial for a successful construction of a properly totally $\alpha$-c.a. set $A$.\\

\noindent This construction outline was used by Downey and Greenberg to build a properly $\alpha$-level set $A$ when $\alpha$ is a power of $\omega$ (\cite{downey2015transfinite} III.L2.2). The authors used a tree in their construction. Because lowness generally doesn't mix on trees, it is not clear if the constructed set is low. In the following Section~\ref{sec:low}, we use the same outline to construct a properly totally $\alpha$-c.a.\ set $A$, but we count injury more carefully to avoid using a tree, thus ensuring the lowness of our set.

%% file: 3.2-low.tex
\subsection{Low, Totally $\alpha$-c.a.} \label{sec:low}
\begin{theorem} \label{thm:low}
    Let $\alpha\leq \epsilon_0$ be a power of $\omega$. Then there exists a c.e.\ set $A$ that is totally $\alpha$-c.a.\ but not totally $\beta$-c.a.\ for any $\beta<\alpha$. Furthermore, $A$ can be made to be low.
\end{theorem}

We use the standard finite injury construction and count injury carefully to bound the number of mind changes of each $\Phi^A_e(e)[-]$ below $\alpha$. Fix a uniformly computable series of functions $\langle\langle f^e_s \rangle_s \rangle_e$ from Fact~\ref{fact:enum} that enumerates the $\beta$-c.a.\ functions for all $\beta<\alpha$. Note that for every $e,x\in\omega$, the limit $f^e(x) :=\lim_s f^e_s(x)$ exists. Also, fix a computable bounding-function $g(e):\omega \to\alpha$ from Fact~\ref{fact:enum} which tells us that $f^e:= \lim_s f^e_s$ is $g(e)$-computably-approximable.\\

\underline{$Q$-module ($(\forall \beta<\alpha)\; [\neg \beta\text{-c.a.}]$):} Build an $A$-computable total function $\Delta^A$ that diagonalizes out of all the $\beta$-c.a.\ functions by satisfying positive requirements $\forall e\in\omega$:
\[Q_e: (\exists x)\; [\Delta^A(x) \neq \lim_s f^e_s(x)].\]

\noindent To meet a single $Q_e$, start by picking a follower $x$ for $Q_e$ and a large use $u$ for $x$. \emph{Diagonalize $\Delta^A$ out of $f^e$} by declaring $\Delta^{A\restriction u}(x)[s] =1-f^e_s(x)$. Wait for $f^e(x)$ to change its mind. Then $Q_e$ will want to \emph{act} by enumerating $u$ into $A$, picking a new use, and diagonalizing $\Delta^{A}$ out of $f^e$ with the new use. $Q_e$ needs permission from negative requirements $N$ before acting. We elaborate later on this permission. $Q_e$ will act if allowed, and will be initialized otherwise. After acting, $Q_e$ will wait for $f^e(x)$ to change its mind again. Then $Q_e$ will want to act again, and we repeat the process of seeking permission from $N$-requirements. Whenever $Q_e$ wants to act, regardless of whether $Q_e$ was allowed, all lower priority $Q>Q_e$ will be initialized.\\

\underline{$N$-module (Totally $\alpha$-c.a., Low):} To make $A$ totally $\alpha$-c.a.\ and low, we construct a partial computable mind change function $\phi(e):\omega\to\alpha$ to satisfy negative requirements $\forall e\in\omega$:
\[N_e: (\exists^\infty s)\; [\Phi^A_e(e)[s]\downarrow] \implies \left[ \phi(e)\downarrow \text{ and } \Phi^A_e(e)[-] \text{ changes mind } \leq \phi(e)\text{-times} \right].\]

\noindent To see how $N_e$ makes $A$ totally $\alpha$-c.a., let $h(e,x)$ be a total computable function such that for all strings $\sigma$ and $e,x\in\omega$, $\Phi^\sigma_e(x) =\Phi^\sigma_{h(e,x)}(h(e,x))$. Then if $\Phi^A_e(-)$ is total, $\phi(h(e,-)): \omega\to\alpha$ will be a total computable bounding function that witnesses $\Phi^A$ being $\alpha$-c.a.. Therefore the $N_e$ requirements above are enough to make $A$ totally $\alpha$-c.a..\\

\noindent To meet a single $N_e$, wait for $\Phi^A_e(e)$ to converge for the first time. We say that $N_e$ is \emph{active}. An active $N_e$ remains active forever, even if $\Phi^A_e(e)$ diverges later. Upon becoming active, $N_e$ decides once and for all which $Q$-requirements are allowed to injure $\Phi^A_e(e)$, and for each such $Q$, how much injury to tolerate. If $N_e$ tolerates $k$ initializations from $Q$, then the injury from $Q$ will be bounded by $g(Q)\cdot (k+1)$. Summing injury from all $Q$ gives a bound $\phi(e)$ of total injury on $N_e$.\\

\noindent For bookkeeping, at each stage $s$ after becoming active, $N_e$ maintains a list $\texttt{Qlist}(e,s)$ of $Q$-requirements allowed to injure $N_e$. Requirements may be removed from the list, but new requirements are never allowed to enter. $N_e$ allows $Q$ to act if $Q$ lies in $N_e$'s list.\\

\underline{Overall Strategy:} When $Q$ acts, $Q$ will enumerate its use into $A$, potentially injuring some active $N$. On the other hand, when $N$ first becomes active, say at stage $s$, $N$ must decide which $Q$ to put into its list, and then decide how many initializations to tolerate from each such $Q$. How can $N$ make these decisions? The overarching strategy is for $N$ to tolerate only injury from the highest priority positive requirement that still wants to act. Let that positive requirement be $Q$. This $Q$ is ``almost stabilized'' since all $Q'<Q$ are presumably satisfied even though $Q$ is not. We argue that the number of times that $Q$ will be initialized is bounded by $k$, the number of active negative requirements at stage $s$:\\

\noindent If $N'$ is active at stage $s$, then $N'$ cannot initialize $Q$ more than once, because after the first initialization, all $Q'>Q$ will be initialized and never allowed to injure $N$ again. On the other hand, if $N''$ is inactive at stage $s$, then $N''$ can never initialize $Q$, because when $N''$ becomes active later, $N''$ can declare to tolerate at least $k$-initializations from $Q$, which is generous enough by earlier argument.\\

\noindent Summarizing, because all $Q'>Q$ are initialized whenever $Q$ acts, $Q$ cannot be initialized more than $k$ times. The overall strategy is for $N$ to tolerate only injury from $Q$. So when $N$ first becomes active, $N$ will put in its list all the $Q$ that have a follower, and allow $k$-initializations from each such $Q$. If some $Q$ in the list is found to be initialized more than $k$-times, or has a higher priority $Q'<Q$ that wants to act, then $Q$ could not have been the amost stabilized requirement, and will never be allowed to injure $N$ again. Then by construction, $N$'s injury is bounded, and from above argument, $Q$'s injury is also bounded, so the construction succeeds.\\

Formally, let the current stage be $s$. Requirement $Q$ is the \emph{almost stabilized positive requirement of stage $s$} if $Q$ is the highest priority positive requirement that still wants to act at or after stage $s$. The $N$ and $Q$ strategies are as follows:\\

\underline{Initialize $N_e$:} Declare $N_e$ inactive.\\

\underline{$N_e$-strategy:} If $N_e$ is not yet active and $\Phi^A_e(e)[s]\uparrow$, then $N_e$ remains inactive and we do nothing. But if $N_e$ is active or $\Phi^A_e(e)[s]\downarrow$, declare $N_e$ active forever. Then if this is the first stage that $N_e$ is active, let
\[k =|\{N: N \text{ is active at stage } s\}|,\]
and set $\phi(e)$ and $\texttt{Qlist}(e,s)$:
\begin{align}
    \label{eq:qlist} \texttt{Qlist}(e,s) &=\{Q: Q \text{ has a follower at stage } s\},\\
    \label{eq:phi} \phi(e) &=\sum_{Q \in \texttt{Qlist}(e,s)}\; g(Q)\cdot (k+1).
\end{align}

\noindent Note that $\phi(e)<\alpha$ from Fact~\ref{fact:sum-power-omega}. In the above equations, if $Q\in\texttt{Qlist}(e,s)$, we say that \emph{$Q$ lies in the quota of $N_e$ at stage $s$}, and we also say that \emph{the quota for $Q$ from $N_e$ is $k$}. If at a later stage $s'>s$, $Q\in\texttt{Qlist}(e,s)$ is initialized more than $k$-times between stages $s$ and $s'$, we say that \emph{$Q$ has exhausted its quota from $N_e$ at stage $s'$}.\\

\noindent Now consider the case when $N_e$ was active before stage $s$, and let $s_0<s$ be the first stage that $N_e$ became active. Then $N_e$ will maintain its list by removing all $Q$ that cannot have been the almost stabilized positive requirement of stage $s_0$. These are the $Q$ requirements that were found to have a higher priority $Q'<Q$ wanting to act after stage $s_0$, or that have exhausted their quota from $N_e$:
\begin{align*}
    \texttt{Qlist}(e,s) &=\texttt{Qlist}(e,s-1)\\
    &-\{Q:\; (\exists Q'<Q)\; [Q' \text{ wanted to act at some stage between } s_0 \text{ and } s]\}\\
    &-\{Q: Q \text{ exhausted its quota from } N_e \text{ at stage } s\}.
\end{align*}

\noindent Finally if $N_e$ is active, $Q$-requirements may ask $N_e$ for permission to act. \emph{$N_e$ allows $Q$ to act at stage $s$} if $Q\in \texttt{Qlist}(e,s)$, or if $\Phi^A_e(e)[s]\uparrow$, or if the use of $Q$ exceeds $\USE(\Phi^A_e(e)[s])$.\\

\underline{Initialize $Q_e$:} Destroy the follower and use assigned to $Q_e$, if any. So the next time $Q_e$ is accessed, $Q_e$ is considered to have neither follower nor use.\\

\underline{$Q_e$-strategy:} If $Q_e$ does not have a follower, assign a new one that is the smallest number not yet in the domain of $\Delta^A$, and assign the follower a new large use. Let $x$ denote the current follower of $Q_e$ and let $u$ denote the current use of $x$. $Q_e$ will want to \emph{diagonalize $\Delta^A$ out of $f^e$} by declaring
$\Delta^{A\restriction u}(x)[s]= 1-f^e_s(x)$. Perform the diagonalization if $\Delta^{A\restriction u}(x)[s]$ was not already declared to be something else earlier, which might occur when $f^e(x)$ changed its mind since the most recent declaration of $\Delta^A(x)$.\\

\noindent Should contradiction occur, \emph{$Q_e$ will want to act} by enumerating $u$ into $A$, picking a new use $u'>u$, then diagonalizing $\Delta^A$ out of $f^e$ with the new use. \emph{$Q_e$ is allowed to act} if every active $N$, even those $N>Q_e$, allows $Q_e$ to act. Refer to the $N$-strategy for when $N$ allows $Q_e$ to act. $Q_e$ will act if allowed and will be initialized otherwise. If initialized, assign $Q_e$ a new follower and use before diagonalizing $\Delta^A$ out of $f^e$. As long as $Q_e$ wanted to act, regardless of whether $Q_e$ was allowed, initialize all $Q>Q_e$.\\

Play the $N$ and $Q$ strategies in the ``usual'' finite injury construction:

\begin{framed}
    At stage $s$, play the $N_0$ to $N_s$ strategies described above to maintain the lists of the active and newly active negative requirements. Then from step $e=0$ to $s$: Play the $Q_e$-strategy described above. If $Q_e$ had a follower and did not want to act, go to the next step. But if $Q_e$ did not have a follower or wanted to act, skip remaining steps and go directly to next stage.
\end{framed}

We now verify that the construction works.

\begin{lemma}
    For all $e\in\omega$, $N_e$ is satisfied. Therefore $A$ is totally $\alpha$-c.a.\ and low.
\end{lemma}
\begin{proof}
    Follows from construction.
\end{proof}

\begin{lemma}
    For all $e\in\omega$, $Q_e$ eventually stops being initialized and stops wanting to act. Therefore $A$ is not totally $\beta$-c.a.\ for all $\beta<\alpha$.
\end{lemma}
\begin{proof}
    It is enough to show that every $Q_e$ eventually stops being initialized, because if $Q_e$ wants to act and is denied permission, then our construction will initialize $Q_e$. We prove by induction on $e\in\omega$.\\
    
    \noindent Wait for $Q_e$ to be almost stabilized. This eventually occurs from induction hypothesis. If $Q_e$ is never initialized again, we are done. So assume $Q_e$ is initialized again. Then from construction, $Q_e$ will be assigned a new follower at the next stage, say at stage $s_0$. Let $\mathcal{N}$ be the set of $N$-requirements that are active at stage $s_0$.\\
    
    \noindent Let $N\in\mathcal{N}$. We first show that $N$ cannot initialize $Q_e$ more than once: If $N$ never injures $Q_e$ again, we are done. So wait for $N$ to injure $Q_e$ again. Then the new use picked by $Q_e$ will exceed the use of $N$, meaning that if $N$ is injured again, that next injury cannot have been inflicted by $Q_e$. Yet the next injury also cannot be inflicted by any $Q>Q_e$, because all $Q>Q_e$ were removed from $\texttt{Qlist}(N,s_0+1)$ when $Q_e$ got initialized just before stage $s_0$. Finally, the next injury also cannot be inflicted by any $Q'<Q_e$ because no $Q'<Q_e$ ever acts again. Therefore computation $N$ is finalized, and the use of $Q_e$ will always exceed the final use of $N$, implying that $N$ will never initialize $Q_e$ again.\\
    
    \noindent Next, we show that if $N'\not\in\mathcal{N}$, then $N'$ can never initialize $Q_e$: Assume for contradiction that there exists a first $N'\not\in\mathcal{N}$ that initializes $Q_e$ after stage $s_0$. Note that after stage $s_0$, $Q_e$ will always have a follower since $Q$-requirements only lose their followers when a higher priority positive requirement gets initialized. Therefore when $N'$ first becomes active, say at stage $s_1>s_0$, $N'$ will put $Q_e$ into its list $\texttt{Qlist}(N',s_1)$, and allow $Q_e$ to be initalized more than $|\mathcal{N}|$-times. Let $s_2\geq s_1$ be the first stage that $N'$ initializes $Q_e$. For $N'$ to be able to initialize $Q_e$, $N'$ must have removed $Q_e$ from its list at some stage between $s_1$ and $s_2$. Since no $Q<Q_e$ acts again, $Q_e$ can only have been removed because $Q_e$ was initialized more than $|\mathcal{N}|$-times since stage $s_1$. Yet from earlier argument, $Q_e$ can only be initialized at most $|\mathcal{N}|$-times by all the elements in $\mathcal{N}$, meaning that the $|\mathcal{N}+1|$-th initialization must have been inflicted by some $N''\not\in\mathcal{N}$, $N''\neq N'$, contradicting the choice of $N'$.\\
    
    \noindent Therefore upon becoming almost stabilized, $Q_e$ cannot be initialized more than $|\mathcal{N}|$-times.
\end{proof}

%% file: 4-nonlow-alpha.tex
\section{Nonlow, Totally $\alpha$-c.a.}
In this section, given $\alpha$ which is a power of $\omega$, we construct a c.e.\ set that is nonlow and totally $\alpha$-c.a.. We use the construction framework of the nonlow, array computable set $A$ in Theorem~\ref{thm:nonlow-ac}. By integrating the $Q$-module of Theorem~\ref{thm:low}, we increase the mind changes of $A$ to the desired $\alpha$.\\

\noindent Array computable sets are totally $\omega$-c.a., which by Facts~\ref{fact:alpha-not-beta} and \ref{fact:alpha-beta}, imply that they sit at the bottom of the mind change hierarchy. Furthermore, by the following definition and fact, not only do array computable sets change their minds less than $\omega$ times, they do so uniformly:

\begin{definition}[\cite{downey2015transfinite} III.D3.3] \label{def:unif-tot-alpha}
    Let $A$ have c.e.\ degree, and let $\alpha<\epsilon_0$. Then $A$ is \emph{uniformly totally $\alpha$-c.a.} if there exists a computable mind change function $m(x):\omega\to\alpha$ such that for every $B\leq_T A$, $B$ has an $\alpha$-c.a.\ approximation $\langle f_s(x),o_s(x)\rangle$ where $o_0(x)\leq m(x)$ for all $x\in\omega$.
\end{definition}

\begin{fact}[\cite{downey2015transfinite} III.L3.4] \label{lemma:ac}
    Let $A$ have c.e.\ degree. Then $A$ is \emph{array-computable} iff $A$ is uniformly totally $\omega$-c.a..
\end{fact}

\noindent Since the set $A$ constructed in Theorem~\ref{thm:nonlow-ac} is nonlow and uniformly totally $\omega$-c.a., roughly speaking, by adding $\alpha$ more mind changes to $A$, $A$'s fickleness will increase to $\alpha$.

\input{4.1-nonlow-construction.tex}
\input{4.2-nonlow-verification.tex}

%% file: 4.1-nonlow-construction.tex
\subsection{Nonlow, Totally $\alpha$-c.a.} \label{sec:nonlow}

\begin{theorem} \label{thm:nonlow}
    Let $\alpha\leq \epsilon_0$ be a power of $\omega$. Then there exists a c.e.\ set $A$ that is totally $\alpha$-c.a.\ but not totally $\beta$-c.a.\ for any $\beta<\alpha$. Furthermore, $A$ can be made to be nonlow.
\end{theorem}

We use the tree framework of Theorem~\ref{thm:nonlow-low2}. To make $A$ nonlow, we use the $P$-module from Theorem~\ref{thm:nonlow-low2}, and to make $A$ not $\beta$-c.a.\ for any $\beta<\alpha$, we use the $Q$-module from Theorem~\ref{thm:low}, but played on a tree. Finally, to make $A$ totally $\alpha$-c.a., we combine the $N$-modules from Theorems~\ref{thm:nonlow-low2} and \ref{thm:low}:\\

\underline{$P$-module (Nonlow):} Exactly the same as the $P$-module of Theorem~\ref{thm:nonlow-low2}.\\

\underline{$Q$-module ($(\forall \beta<\alpha)\; [\neg \beta\text{-c.a.}]$):} Same as $Q$-module of Theorem~\ref{thm:low}, but played on a tree. We elaborate on the subtle differences from the finite injury construction later.\\

\noindent Note that even though $P$ and $Q$ both want to enumerate elements into $A$ to diagonalize out of the limit of functions, their module designs are very different for two reasons. First, the function $\Delta^A$ constructed by $Q$ needs to be total, but the function $\Gamma^A$ constructed by $P$ can be partial. Therefore when picking a new use, $P$ can wait for negative requirements $N$ to stabilize sufficiently first, but $Q$ cannot because $Q$ may never know when and if stability will occur. The moment $Q$ has a follower $x$, $Q$ must immediately declare that $\Delta^{A\restriction u}(x)$ converges with some use $u$, otherwise $\Delta^A(x)$ may diverge at the end of the construction. Because of $Q$'s impatience, $Q$ will injure $N$ requirements more often than $P$ does since $Q$ may pick uses that are too small to avoid inflicting future injuries.\\

\noindent But $Q$ redeems itself by informing $N$ requirements of the injury $g(Q)<\alpha$ to expect. $P$ on the other hand, cannot bound its injury computably. $Q$'s lack of pacing but increased informativeness explains the stark differences between the $P$ and $Q$ strategies introduced earlier.\\

\underline{$N$-module (Totally $\alpha$-c.a.):} For all $e\in\omega$, construct mind change function $\phi_e(-):\omega\to\alpha$ satisfying:
\[N_e: \Phi^A_e\; \text{total} \implies \left[\phi_e(-) \text{ total computable and } (\forall x)\; \left[\Phi^A_e(x)[-] \text{ changes mind } \leq \phi_e(x)\text{-times}\right] \right].\]

\noindent For every $x$, $N(x)$ works to protect computation $\Phi^A_e\restriction x$. Yet $N(x)$ must also offer $P$ and $Q$ requirements sufficient opportunities to act. Following the argument of the $N$ module in Theorem~\ref{thm:nonlow-low2}, $N(x)$ allows $P$ requirements to act via the quota system. Also, following the argument of the $N$ module in Theorem~\ref{thm:low} and Eq.~(\ref{eq:qlist}), $N(x)$ allows $Q$ requirements to act via the $\texttt{Qlist}$ system, but modified to fit a tree construction. We elaborate on the modifications later.\\

\noindent Note that unlike the $N$ requirements of Theorem~\ref{thm:low}, the series of mind change functions $\phi_e$ will not be uniformly computable because we are not trying to satisfy lowness. However for a fixed $e$, if $\Phi^A_e$ is total, then $\phi_e(-)$ must be total-computable.\\

\underline{Tree construction:} Order the requirements $N_0,P_0,Q_0,N_1,P_1,Q_1,\ldots$ so that a node $\delta$ in the tree $\Lambda$ works for $N_e$ if $|\delta|=3e$, works for $P_e$ if $|\delta|=3e+1$, and works for $Q_e$ if $|\delta|=3e+2$. Given $\delta\in\Lambda$, if $\delta$ works for $N_e$ we write $\delta=\eta_e$, and if $\delta$ works for $P_e$ we write $\delta=\rho_e$, and if $\delta$ works for $Q_e$ we write $\delta=\xi_e$. $\eta$ and $\rho$ nodes have two possible outcomes $\{\infty<\texttt{fin}\}$, while $\xi$ nodes, being of finite injury, have only one possible outcome.\\

\noindent Let the current stage be $s$. Given a node $\delta\in\Lambda$, depending on the type of requirement $\delta$ is working for, we initialize and design the $\delta$-strategy as follows:\\

\underline{Initialize $\rho$:} Destroy the follower and use assigned to $\rho$, if any. So the next time $\rho$ is visited, $\rho$ is considered to have neither follower nor use.\\

\underline{Outcome of $\rho$:} Exactly the same as the $\rho$-outcome of Theorem~\ref{thm:nonlow-low2}.\\

\underline{$(\rho^\frown\texttt{fin})$-strategy:} Exactly the same as the $(\rho^\frown\texttt{fin})$-strategy of Theorem~\ref{thm:nonlow-low2}.\\

\underline{$(\rho^\frown\infty)$-strategy:} Exactly the same as the $(\rho^\frown\infty)$-strategy of Theorem~\ref{thm:nonlow-low2}.\\

\underline{Initialize $\xi$:} Destroy the follower and use assigned to $\xi$, if any. So the next time $\xi$ is visited, $\xi$ is considered to have neither follower nor use.\\

\underline{Outcome of $\xi$:} $\xi$ only has one outcome since it inflicts only finite injury.\\

\underline{$\xi$-strategy:} Same as the $Q$-strategy in Theorem~\ref{thm:low}, except that if $\xi$ wants to act, $\xi$ only needs permission from all initial segments $\eta$ such that $\eta^\frown\infty \preceq\xi$. Permission is no longer needed from lower priority $\eta\succeq\xi$ because unlike the construction in Theorem~\ref{thm:low}, we are not trying to make the mind change function $\phi_e(-)$ to be uniformly computable.\\

\noindent If $\xi$ wants to act, initialize all $\xi'\succ\xi$ and $\delta>_L\xi$. If $\xi$ wants to act but was not allowed, initialize $\xi$, then assign $\xi$ new follower and large use, before diagonalizing $\Delta^A$ out of $f^\xi$. Given $k\in\omega$, we let $\xi(k)$ denote the sub-requirement of $\xi$ that works to let $\xi$ act $k$ times since the beginning of the construction.\\

\underline{Initialize $\eta$:} To keep track of the injury from $\rho$ nodes, for all $x\in\omega$, $\eta(x)$ is fixed a quota $\texttt{quota}(x)$, which is identical to Eq.~(\ref{eq:quota}) of Theorem~\ref{thm:nonlow-low2}. Also, to keep track of injury from $\xi$ nodes, for each $x,s'\in\omega$, $\eta$ maintains a list $\texttt{Qlist}(\eta,x,s')$ of $\xi$ nodes allowed to injure $\eta(x)$ at stage $s'$. Initialize $\texttt{Qlist}(\eta,x,s)$ to undefined for all $x\in\omega$.\\

\underline{Outcome of $\eta$:} Exactly the same as the $\eta$-outcome of Theorem~\ref{thm:nonlow-low2} -- $\eta$ has outcome $\infty$ iff $s$ is $\eta$-expansionary. Note that the definition of $\eta$-correctness considers only the $\rho$ nodes that are initial segments of $\eta$. In particular, the correctness of $\eta(x)$ will not be affected by the use of some $\xi\prec\eta$ being smaller than $\USE(\eta(x)[s])$, because $\xi$ nodes are finitary nodes, and in tree constructions, when accessing $\delta\succeq\xi$, we always assume that $\xi$ never acts again.\\

\underline{$(\eta^\frown\texttt{fin})$-strategy:} Do nothing.\\

\underline{$(\eta^\frown\infty)$-strategy:} Some $\rho \succeq \eta^\frown\infty$ might ask $\eta$ for permission to act via picking use or via enumeration. $\eta$'s strategy for granting permission is exactly the same as the $(\eta^\frown\infty)$-strategy of Theorem~\ref{thm:nonlow-low2}. Note again that the $\xi$ nodes are irrelevant when determining $\rho$-correctness of $\eta(x)$ computations.\\

\noindent Also, some $\xi\succeq \eta^\frown\infty$ may ask for permission to act. $\eta$'s strategy for granting permission is similar to the $N$-strategy of Theorem~\ref{thm:low}, but modified for a tree construction. The main idea is still for $\eta$ to tolerate only actions from the ``almost stabilized'' $\xi$ node. Formally, the \emph{the almost stabilized node of stage $s$} is the unique node $\xi$ such that $\xi\prec\delta_\omega$, $\xi$ is the highest priority node that still wants to act at or after stage $s$, and such that no $\delta\prec\xi$ is initialized again after stage $s$. Like before, we can computably bound the number of initializations from the almost stabilized $\xi$. However this bound is more difficult to calculate because of the tree construction, so we only sketch the idea here, and flesh out the details later in Main Lemma 2.\ref{n:Q-init2}:\\

\noindent Let $\eta'^\frown\infty \preceq \xi$, $x<l_s(\eta')$, and consider the number of times that $\eta'(x)$ will initialize the almost stabilized $\xi$. Wait for $\xi$ to be initialized again. Then all $\xi'\succ\xi$ and $\xi'>_L\xi$ will be initialized and never allowed to injure $\eta'(x)$ again. Then by near stability of $\xi$, the only positive nodes that can injure $\eta'(x)$ are $\xi$ itself and those $\rho\in\texttt{quota}(x)$. So after $\xi$ exhausts its quota from $\eta'(x)$, $\eta'(x)$ will behave like in the proof of Theorem~\ref{thm:nonlow-ac}, where there are no $Q$ nodes to complicate the injury count of $\eta'(x)$. Thus the number of times that $\eta'(x)$ gets injured will be bounded by $(x+1)^24^{(x+1)^2}$, as computed in Eq.~(\ref{eq:injpow-tot}), implying that $\xi$ cannot be initialized by $\eta'(x)$ more than $(x+1)^24^{(x+1)^2}$ times. Summing the initializations from $\eta'(x)$ for all $\eta'^\frown\infty \preceq\xi$ and all $x<l_s(\eta')$ gives a computable bound 
\begin{equation}
    \label{eq:k2} k'(\xi,s) :=\sum_{\substack{\eta'^\frown\infty \preceq \xi,\\ x<l_s(\eta')}} (x+1)^24^{(x+1)^2}
\end{equation}
of total initializations of $\xi$ after stage $s$.\\

\noindent Like in the finite injury construction, $\eta(x)$ maintains a list $\texttt{Qlist}(\eta,x,s)$ of $\xi$ nodes that are potentially the almost stabilized node of stage $s$, routinely removing nodes that are later found to not have been the almost stabilized node. And for each $\xi$ in the list, $\eta(x)$ will assume that $\xi$ is the almost stabilized node, and tolerate injury from $\xi$ up to the $k'$ bound computed earlier.\\

\noindent Formally, for every $x<l_s(\eta)$, we set or update $\texttt{Qlist}(\eta,x,s)$ as follows: If the current stage $s$ is the first $(\eta^\frown\infty)$-stage, or if there was a most recent $(\eta^\frown\infty)$-stage $s_0<s$ but $\eta$ was initialized between stage $s_0$ and $s$, or if $\texttt{Qlist}(\eta,x,s_0)$ has not yet been defined, set
\begin{align}
    \label{eq:qlist2} \texttt{Qlist}(\eta,x,s) &=\{\xi\succeq \eta^\frown\infty: \xi \text{ has a follower at stage } s\}.
\end{align}

\noindent Then following earlier argument, $\eta(x)$ will allow $k(\eta(x),s)$ initializations from each $\xi\in\texttt{Qlist}(\eta,x,s)$, where
\begin{align}
    \label{eq:k3} k(\eta(x),s) :=\max \left\{k'(\xi,s):\; \xi\in\texttt{Qlist}(\eta,x,s)\right\}.
\end{align}

\noindent For each $\xi\in\texttt{Qlist}(\eta,x,s)$, we say that \emph{$\xi$ lies in the quota of $\eta(x)$ at stage $s$}, and that \emph{the quota for $\xi$ from $\eta(x)$ is $k(\eta(x),s)$}. If at a later stage $s'>s$, $\xi\in\texttt{Qlist}(\eta,x,s)$ is initialized more than $k(\eta(x),s)$-times between stages $s$ and $s'$, then we say that \emph{$\xi$ has exhausted its quota from $\eta(x)$ at stage $s'$}.\\

\noindent Now if $\eta$ was not initialized between the most recent $(\eta^\frown\infty)$-stage $s_0<s$ and the current stage $s$, and $\texttt{Qlist}(\eta,x,s_0)$ is already defined, then update $\eta(x)$'s list by removing all the $\xi$ nodes that cannot have been the almost stabilized node of stage $s_0$. These are the $\xi$ nodes that have either exhausted their quota from $\eta(x)$, or have been initialized by some positive node $\prec\xi$ or by some $\delta<_L\xi$:
\begin{align} \label{eq:qlist-update}
    \texttt{Qlist}(\eta,x,s) &=\texttt{Qlist}(\eta,x,s_0)\\ \nonumber
    &-\{\xi:\; \xi \text{ exhausted its quota from } \eta(x) \text{ at stage } s\}\\ \nonumber
    &-\{\xi:\; (\exists \xi'\prec\xi)\; [\xi' \text{ wanted to act at some stage between } s_0 \text{ and } s]\}\\ \nonumber
    &-\{\xi:\; (\exists \rho\prec\xi)\; [\rho \text{ was initialized between stages } s_0 \text{ and } s]\}\\ \nonumber
    &-\{\xi:\; (\exists \delta<_L\xi)\; [\text{there is a } \delta\text{-stage between } s_0 \text{ and } s]\} \nonumber
\end{align}

\noindent Finally, given $x<l_s(\eta)$ and $\xi\succeq\eta^\frown\infty$, we say that \emph{$\eta(x)$ allows $\xi$ to act at stage $s$} if $\xi\in \texttt{Qlist}(\eta,x,s)$, or if $\Phi^A_\eta(x)[s]\uparrow$, or if $\xi$'s use exceeds the $\text{use} \left(\Phi^A_\eta(x)[s] \right)$. Then, \emph{$\eta$ allows $\xi$ to act at stage $s$} if for all $x<l_s(\eta)$, $\eta(x)$ allows $\xi$ to act at stage $s$.\\

Play the $\rho$, $\xi$, and $\eta$ strategies on a tree to construct the nonlow, totally $\alpha$-c.a.\ set $A$:

\begin{framed}
    \underline{Stage $s$:} Let $\delta_{s,0}$ be the empty node. From step $e=0$ to $s$: Determine the outcome $o$ of $\delta_{s,e}$. Set $\delta_{s,e+1} =\delta_{s,e}^\frown o$, and initialize all nodes to the right of $\delta_{s,e+1}$. If $\delta_{s,e}^\frown o=\rho^\frown\texttt{fin}$, play the $(\rho^\frown\texttt{fin})$-strategy to diagonalize $\Gamma^A$ out of $\psi^\rho$. If $\delta_{s,e}^\frown o =\eta^\frown\infty$, play the $(\eta^\frown\infty)$-strategy described above to set or to maintain $\texttt{Qlist}(\eta,x,s)$.\\
    
    \noindent At the end of step $s$, we will get a node $\delta_s :=\delta_{s,s} \in\Lambda$ of length $s$. Some of the initial segments belonging to positive nodes may want to act. Let
    \begin{align*}
        \Theta &=\{\rho\preceq\delta_s:\; \rho^\frown\infty \preceq \delta_s, \text{ and } \rho \text{ wants to pick use at stage } s \text{ and is allowed}\}\\
        &\cup\{\rho\preceq\delta_s:\; \rho^\frown\infty \preceq \delta_s, \text{ and } \rho \text{ wants to enumerate use at stage } s\}\\
        &\cup\{\xi\preceq\delta_s:\; \xi \text{ wants to act at stage } s\}.
    \end{align*}
    
    \noindent If $\Theta$ is empty, go to the next stage. Otherwise, select one $\delta\in\Theta$. This $\delta$ will be called \emph{the selected node at stage $s$}, and is selected as:
    \[\delta =\underset{\delta'\in\Theta}{\mathrm{argmin}}\; \{\langle \delta',k\rangle:\; \delta' \text{ has been selected } k \text{-times before stage } s\}.\]
    If $\delta=\xi$, play the $\xi$-strategy described above. If $\delta=\rho$, play the $(\rho^\frown\infty)$-strategy described above. Go to next stage.
\end{framed}

%% file: 4.2-nonlow-verification.tex
\begin{lemma} \label{lemma:delta-omega2}
    $\delta_\infty \in[\Lambda]$.
\end{lemma}
\begin{proof}
    Same as proof of Lemma~\ref{lemma:delta-omega}.
\end{proof}

\begin{mainlemma2*}
    Given $n\in\omega$, let $\delta=\delta_\omega\restriction n \in \Lambda$.
    \begin{enumerate}
        \item \label{n:P-init2} If $\delta=\rho$ then $\rho$ eventually stops being initialized. Thus $\rho$ has a final follower $y_\rho$.
        \item \label{n:P-fin2} If $\delta=\rho$ and $\rho^\frown\texttt{fin} \prec\delta_\omega$, then if $\lim_s \psi^\rho_s(y_\rho)$ exists the limit will not equal $\chi_{\Gamma^A}(y_\rho)$.
        \item \label{n:P-inf2} If $\delta=\rho$ and $\rho^\frown\infty \prec\delta_\omega$, then $\rho$ will act infinitely often. Thus $\lim_s \psi^\rho_s(y_\rho)$ does not exist.
        \item \label{n:Q-init2} If $\delta=\xi$ then $\xi$ eventually stops being initialized. Thus $\xi$ has a final follower $z_\xi$.
        \item \label{n:Q} If $\delta=\xi$ then $\xi$ eventually stops wanting to act. Thus $\lim_s f^\xi(z_\xi) \neq \Delta^A(z_\xi)$.
        \item \label{n:N-init2} If $\delta=\eta$ then $\eta$ eventually stops being initialized.
        \item \label{n:N-fin2} If $\delta=\eta$ and $\eta^\frown\texttt{fin} \prec\delta_\omega$, then $\Phi^A_\eta$ is not total.
        \item \label{n:N-inf2} If $\delta=\eta$ and $\eta^\frown\infty \prec\delta_\omega$, then $\Phi^A_\eta$ is total, and there is computable function $\phi_\eta(-):\omega\to\alpha$ such that for all $x$, $\Phi^A_\eta(x)[s]$ changes its mind less than $\phi_\eta(x)$ times during the $\eta$-stages.
    \end{enumerate}
\end{mainlemma2*}

\vspace{2em}
\noindent We prove Main Lemma 2 after this immediate corollary:

\begin{corollary} \label{cor:A-nonlow-alpha}
    $A$ is nonlow and properly totally $\alpha$-c.a..
\end{corollary}
\begin{proof}
    Let $e\in\omega$ be arbitrary. By the same argument as Corollary~\ref{cor:A-nonlow-low2}, $P_e$ is satisfied. From Main Lemma 2.\ref{n:Q-init2} and 2.\ref{n:Q}, $Q_e$ is satisfied. Finally, from Lemmas~\ref{lemma:delta-omega2}, Main Lemma 2.\ref{n:N-fin2} and 2.\ref{n:N-inf2}, $N_e$ is satisfied.
\end{proof}

\vspace{2em}
We prove the \ref{n:N-inf2} claims of Main Lemma 2 by simultaneous induction on $n$. Like in the verification of Main Lemma 1, in each claim for $\delta\in\Lambda$, always assume that we are working in $\delta$-stages after higher priority requirements or sub-requirements have stabilized. These stages exist from induction hypothesis. Refer to the proof of Main Lemma 1 for what it means for $\rho$, $\eta$, or $\eta(x)$ requirements to have stabilized at stage $s$. As for $\xi$ nodes, \emph{$\xi$ has stabilized at stage $s$} if $\xi$ never wants to act again at or after stage $s$. If $\delta=\rho$ or $\delta=\xi$, then given $k\in\omega$, sub-requirement \emph{$\delta(k)$ has stabilized at stage $s$} if for all $\langle \delta',k'\rangle <\langle \delta,k\rangle$, if $\delta'$ is ever selected $k'$ times or less since the beginning of the construction, then these selections have already been made by stage $s$.\\

\begin{claim*}[Main Lemma 2.\ref{n:P-init2}]
    If $\rho\prec \delta_\omega$, then $\rho$ eventually stops being initialized.
\end{claim*}
\begin{proof}
    Same as the proof for Main Lemma 1.\ref{n:P-init}. $\xi$ nodes do not affect the argument.
\end{proof}

\begin{claim*}[Main Lemma 2.\ref{n:P-fin2}]
    If $\rho^\frown\texttt{fin} \prec \delta_\omega$, then if $\lim_s \psi^\rho_s(y_\rho)$ exists the limit will not equal $\chi_{\Gamma^A} (y_\rho)$.
\end{claim*}
\begin{proof}
    Same as the proof for Main Lemma 1.\ref{n:P-fin}. $\xi$ nodes do not affect the argument.
\end{proof}

\begin{claim*}[Main Lemma 2.\ref{n:P-inf2}]
    If $\rho^\frown\infty \prec \delta_\omega$, then $\rho$ will act infinitely often.
\end{claim*}
\begin{proof}
    Same as the proof for Main Lemma 1.\ref{n:P-inf}. $\xi$ nodes do not affect the argument.
\end{proof}

\begin{claim*}[Main Lemma 2.\ref{n:Q-init2}]
    If $\xi\prec \delta_\omega$, then $\xi$ eventually stops being initialized.
\end{claim*}
\begin{proof}
    Wait for $\xi$ to be the almost stabilized node, i.e.\ wait for all $\delta\prec\xi$ to stabilize. If $\xi$ is never initialized, we are done. So wait for $\xi$ to be initialized, say at stage $s_0$. We show that $\xi$ cannot be initialized more than $k'(\xi,s_0)$-times after stage $s_0$, and where $k'(\xi,s_0)$ is defined in Eq.~(\ref{eq:k2}). Assume for contradiction that $\xi$ was initialized more than $k'(\xi,s_0)$ times. Now any initialization of $\xi$ at or after stage $s_0$ must be due to $\xi$ being denied action by some $\eta(x)$ with $\eta^\frown\infty \preceq\xi$, because nodes $\delta<_L\xi$ are never visited and nodes $\delta\prec\xi$ are never initialized by near stability of $\xi$.\\
    
    \noindent First consider the case where the $[k'(\xi,s_0)+1]$-th intialization of $\xi$ was due to $\eta(x)$ with $\eta^\frown\infty \preceq\xi$ and $x>l_{s_0}(\eta)$. Since $\xi$ is always initialized from being denied action, the $\xi$-strategy will ensure that $\xi$ always has a follower. Therefore at the $(\eta^\frown\infty)$-stage $s_1$ when $x<l_{s_1}(\eta)$ for the first time after stage $s_0$, $(\eta^\frown\infty)$-strategy will put $\xi$ into $\eta(x)$'s $\texttt{Qlist}$. Furthermore, from Eqs.~(\ref{eq:k2}) and (\ref{eq:k3}), $\eta(x)$ will tolerate $k(\eta(x),s_1) \geq k'(\xi,s_1) >k'(\xi,s_0)$ initializations from $\xi$, as long as $\xi$ is not removed from $\eta(x)$'s $\texttt{Qlist}$ before $\xi$ exhausts $\eta(x)$'s quota. But premature removal is not possible from Eq.~(\ref{eq:qlist-update}), because $\xi$ is almost stabilized.\\
    
    \noindent So it must be that $[k'(\xi,s_0)+1]$-th intialization of $\xi$ was due to $\eta(x)$ with $\eta^\frown\infty \preceq\xi$ and $x\leq l_{s_0}(\eta)$. But $\eta(x)$ cannot initialize $\xi$ so often: At the end of stage $s_0$, the $\xi$-strategy will initialize all $\xi'\succ\xi$ and $\xi'>_L\xi$. Then immediately at the next $(\eta^\frown\infty)$-stage, the $(\eta^\frown\infty)$-strategy will remove these $\xi'$ from $\eta(x)$'s $\texttt{Qlist}$ to stop them from ever injuring $\eta(x)$ again. So apart from $\xi$ itself, the only $\xi''$ nodes that can injure $\eta(x)$ are those $\xi''\succeq \eta^\frown\infty$ with $\xi''\prec\xi$. But these $\xi''$ never want to act again by near stability of $\xi$, implying they can never injure $\eta(x)$. Thus the only positive nodes that can injure $\eta(x)$ are $\xi$ itself and those $\rho\in\texttt{quota}(x)$. Now before $\xi$ exhausts quota from $\eta(x)$, $\eta(x)$ cannot injure $\xi$. So wait for quota to be exhausted. Then $\xi$ will also stop injuring $\eta(x)$, so $\eta(x)$ will behave like in the proof of Theorem~\ref{thm:nonlow-ac}, with no $\xi'$ requirements to complicate counting of injury. Then by Eq.~(\ref{eq:injpow-tot}), $\eta(x)$ cannot be injured more than $(x+1)^24^{(x+1)^2}$ times, implying that $\eta(x)$ cannot initialize $\xi$ more than $(x+1)^24^{(x+1)^2}$ times, which is fewer times than $k'(\xi,s_0)$ from Eq.~(\ref{eq:k2}).
\end{proof}

\begin{claim*}[Main Lemma 2.\ref{n:Q}]
    If $\xi\prec \delta_\omega$, then $\xi$ eventually stops wanting to act.
\end{claim*}
\begin{proof}
    Follows from the fact that $\xi$ always will want to act the moment $\Delta^A$ is not diagonalized out of $f^\xi$ via the final follower $z_\xi$, and the fact that $\lim_s f^\xi_s(z_\xi)$ exists, and from induction hypothesis with Main Lemma 2.\ref{n:Q-init2}.
\end{proof}

\begin{claim*}[Main Lemma 2.\ref{n:N-init2}]
    If $\eta\prec \delta_\omega$, then $\eta$ eventually stops being initialized.
\end{claim*}
\begin{proof}
    Wait for all $\rho, \xi\prec\eta$ to stabilize. Then $\eta$ will never be initialized again.
\end{proof}

\begin{claim*}[Main Lemma 2.\ref{n:N-fin2}]
    If $\eta^\frown\texttt{fin} \prec \delta_\omega$, then $\Phi^A_\eta$ is not total.
\end{claim*}
\begin{proof}
    Same as the proof for Main Lemma 1.\ref{n:N-fin}.
\end{proof}

\begin{claim*}[Main Lemma 2.\ref{n:N-inf2}]
    If $\eta^\frown\infty \prec\delta_\omega$, then $\Phi^A_\eta$ is total, and there is computable function $\phi_\eta(-):\omega\to\alpha$ such that for all $x$, $\Phi^A_\eta(x)[s]$ changes its mind less than $\phi_\eta(x)$ times during the $\eta$-stages.
\end{claim*}

\noindent Like in Main Lemma 1.\ref{n:N-inf} of Theorem~\ref{thm:nonlow-low2}, this claim is the heart of the argument on why the construction works. Define $S_\eta$, $t_\eta(x)$, and $a_{\eta,s}(x)$ like in Eqs.~(\ref{eq:S}), (\ref{eq:t}), and (\ref{eq:g}). Note that as before, for fixed $\eta$, these sets and functions are computable, but not uniformly so. For a given $x\in\omega$, always assume we are working in the $S_\eta$-stages after $t_\eta(x)$.\\

\noindent We prove the claim by induction on $x$. We want to show that for all $x\in\omega$, $\eta(x)$ gets injured by positive nodes less than $\alpha$ times, and we can compute this injury bound $\phi_\eta(x)$ as a computable function of $x$. The proof outline follows the proof of Theorem~\ref{thm:nonlow-ac}, but with additional injuries from $\xi$ nodes.\\

\noindent From construction, at stage $t_\eta(x)$, $\eta(x)$ will set $\texttt{Qlist}(\eta,x,t_\eta(x))$ according to Eq.~(\ref{eq:qlist2}), deciding once and for all which $\xi$ nodes are allowed to injure $\eta(x)$. Then by Eq.~(\ref{eq:k3}), $\eta(x)$ also decides to tolerate $k(\eta(x), t_\eta(x))$ initializations from each such $\xi$ node. Between intializations, $\xi$ cannot inflict more than $g(\xi)$ injury, therefore the total injury on $\eta(x)$ from all the $\xi$ nodes cannot exceed
\begin{align} \label{eq:beta}
    \beta(\eta(x)) := \sum_{\xi\in \texttt{Qlist}(\eta,x,t_\eta(x))} g(\xi) \cdot (k(\eta(x), t_\eta(x)) +1),
\end{align}
which is less than $\alpha$ from Fact~\ref{fact:sum-power-omega}.\\

\noindent Now consider the injury from $\rho$ nodes. If $\rho$ can injure $\eta(x)$, then from Fact~\ref{fact:inj-set}, $\rho$ must lie in the quota $\texttt{quota}(x)$ of $\eta(x)$. And there are only finitely many such $\rho$ from Eq.~(\ref{eq:quota}). Recall that in Theorem~\ref{thm:nonlow-ac}, the main proof idea was that after $\rho$ exhausts quota, $\rho$ can only injure $\eta(x)$ if triggered by some $\rho'$ of smaller edge layer. This claim still holds except that $\rho$ might also have been triggered by $\xi$ nodes:

\begin{claim} \label{claim:trigger2}
    If $\delta$ is the node that triggered $\rho$ to injure $\eta(x)$ at stage $s$, then either $\delta=\xi\in \texttt{Qlist}(\eta,x,t_\eta(x))$, or $\delta=\rho'$ where $\rho'\succeq \rho^\frown\infty$. In particular, if $\delta=\rho'$, then $\rho'$ has smaller edge layer than $\rho$.
\end{claim}
\begin{proof}
    If $\delta\neq\xi$, then the same argument for Claim~\ref{claim:trigger} works here too.
\end{proof}

\begin{claim} \label{claim:edge-inductive2}
    If $\rho\in\texttt{quota}(x)$ has edge layer $r$, then
    \begin{align} \label{eq:injpow2}
        \texttt{InjPow}_\eta(x,\rho) \leq \texttt{quotaFor}_x(\rho) +\beta(\eta(x)) +\sum_{\substack{\rho'\in \texttt{quota}(x)\\ \rho' \text{ has edge layer } <r}}\; \texttt{InjPow}_\eta(x,\rho').
    \end{align}
\end{claim}
\begin{proof}
    The $\texttt{quotaFor}_x(\rho)$ term comes from $\rho$ exhausting quota from $x$, and the $\beta(\eta(x))$ and summation terms come from Claim~\ref{claim:trigger2}.
\end{proof}

\begin{proof}
    (Of Main Lemma 2.\ref{n:N-inf2}): From Fact~\ref{fact:inj-set}, working only in the $(\eta^\frown\infty)$-stages after stage $t_\eta(x)$,
    \begin{equation*}
        \text{Number of times } \eta(x) \text{ gets injured } \leq \beta(\eta(x)) +\sum_{\rho\in\texttt{quota}(x)} \texttt{InjPow}_\eta(x,\rho).
    \end{equation*}
    $\beta(\eta(x))$ term is computable as a function of $x$. Also, each term within the summation is computably bounded below $\alpha$ by the recursive relation Eq.~(\ref{eq:injpow2}) and Fact~\ref{fact:sum-power-omega}. Together with the fact that $\texttt{quota}(x)$ is a computable set and Fact~\ref{fact:sum-power-omega} again, the total injury on $\eta(x)$ will be computable and less than $\alpha$.\\
    
    \noindent To give a concrete bound on injury, like in the proof of Theorem~\ref{thm:nonlow-ac}, we assume that the pairing function used for the definition of $\texttt{quota}(x)$ is such that
    \[\max \{\max(|\rho|, k):\; \langle \rho,k\rangle \in \texttt{quota}(x)\} < x.\]
    We repeat the inductive proof of Theorem~\ref{thm:nonlow-ac}, but with the role of $\texttt{quota}(x)$ replaced by $\beta(\eta(x))+\texttt{quota}(x)$. Then given $\rho\in\texttt{quota}(x)$ with edge layer $r$, $\texttt{InjPow}_\eta(x,\rho)$ will be bounded by $[\beta(\eta(x)) +x]\cdot (r+1)4^{(r+1)^2}$, so
    \begin{align*} \label{eq:injpow-tot2}
        \text{Number of times } \eta(x) \text{ gets injured} \leq [\beta(\eta(x)) +x]\cdot (x+1)4^{(x+1)^2}.
    \end{align*}
\end{proof}

%% file: main.bbl
\begin{thebibliography}{10}

\bibitem{downey1993array}
Rod Downey.
\newblock Array nonrecursive degrees and lattice embeddings of the diamond.
\newblock {\em Illinois journal of mathematics}, 37(3):349--374, 1993.

\bibitem{downey2015transfinite}
Rod Downey and Noam Greenberg.
\newblock A transfinite hierarchy of lowness notions in the computably
  enumerable degrees, unifying classes, and natural definability, 2015.

\bibitem{downey2007totally}
Rod Downey, Noam Greenberg, and Rebecca Weber.
\newblock Totally $\omega$-computably enumerable degrees and bounding critical
  triples.
\newblock {\em Journal of Mathematical Logic}, 7(02):145--171, 2007.

\bibitem{downey1990array}
Rod Downey, Carl Jockusch, and Michael Stob.
\newblock Array nonrecursive sets and multiple permitting arguments.
\newblock In {\em Recursion theory week}, pages 141--173. Springer, 1990.

\bibitem{downey1996array}
Rod Downey, Carl Jockusch, and Michael Stob.
\newblock Array nonrecursive degrees and genericity.
\newblock {\em Computability, Enumerability, Unsolvability (Cooper, Slaman,
  Wainer, eds.), London Mathematical Society Lecture Notes Series},
  224:93--105, 1996.

\bibitem{downey1987t}
Rodney~G Downey and Carl~G Jockusch.
\newblock T-degrees, jump classes, and strong reducibilities.
\newblock {\em Transactions of the American Mathematical Society},
  301(1):103--136, 1987.

\bibitem{downey1997contiguity}
Rodney~G Downey and Steffen Lempp.
\newblock Contiguity and distributivity in the enumerable turing degrees.
\newblock {\em The Journal of Symbolic Logic}, 62(4):1215--1240, 1997.

\bibitem{ishmukhametov1999weak}
Shamil Ishmukhametov.
\newblock Weak recursive degrees and a problem of spector.
\newblock {\em Recursion theory and complexity (Kazan, 1997)}, 2:81--87, 1999.

\bibitem{lachlan1972embedding}
Alistair~H Lachlan.
\newblock Embedding nondistributive lattices in the recursively enumerable
  degrees.
\newblock In {\em Conference in Mathematical Logic-London'70}, pages 149--177.
  Springer, 1972.

\bibitem{soare1978recursively}
Robert~I Soare.
\newblock Recursively enumerable sets and degrees.
\newblock {\em Bulletin of the American Mathematical Society},
  84(6):1149--1181, 1978.

\end{thebibliography}
